\documentclass[11pt,a4paper]{article}
\pdfoutput=1

\usepackage[utf8]{inputenc}
\usepackage[T1]{fontenc}
\usepackage[english]{babel}
\usepackage{microtype}
\usepackage{amsmath}
\usepackage{amsfonts}
\usepackage{amssymb}
\usepackage{amsthm}
\usepackage{amscd}
\usepackage{latexsym}
\usepackage{thmtools}
\usepackage{lmodern}
\usepackage{stmaryrd}
\usepackage{MnSymbol}
\usepackage{graphicx}
\usepackage{extarrows}
\usepackage{mathtools}
\usepackage[center]{caption}
\usepackage{subfig}
\usepackage{mathrsfs}
\usepackage{mathtools}
\usepackage{slashed}
\usepackage{amsthm}
\usepackage{pifont}
\usepackage{fancyhdr}

\usepackage{hyperref}
\hypersetup{
    colorlinks=true,
    linkcolor=blue,
    urlcolor=RoyalBlue,
    citecolor=ForestGreen
    }
    
\usepackage[noabbrev,capitalize,nameinlink]{cleveref}
\crefname{prop}{Proposition}{Propositions}
\crefname{fact}{Fact}{Facts}

\usepackage[shortlabels]{enumitem} 
\setlist[enumerate,1]{label={\upshape\alph*)}}
\setlist[enumerate,2]{label={\upshape\roman*)}}

\usepackage{xfrac}
\usepackage{graphicx}
\usepackage[dvipsnames]{xcolor} 
\usepackage{slashed}
\usepackage[autostyle=true]{csquotes}
\usepackage[all,cmtip]{xy}
\usepackage{cancel} 
\usepackage{xparse}

\usepackage[a4paper,
            top=2.5cm,
            bottom=2.5cm,
            inner=3.5cm,
            outer=3.5cm]{geometry}

\usepackage[citestyle=alphabetic,
            bibstyle=alphabetic,
            backend=bibtex,
            url=true,
            doi=true,
            isbn=false,
            giveninits=true]{biblatex}

\bibliography{secondary_index}
\AtEveryBibitem{%
  \ifentrytype{article}{
    \clearfield{url}%
    \clearfield{urldate}%
  }{}
  \ifentrytype{book}{
    \clearfield{url}%
    \clearfield{urldate}%
  }{}
  \ifentrytype{incollection}{
    \clearfield{url}%
    \clearfield{urldate}%
  }{}
} 

\declaretheorem[name=Theorem,numberwithin=section]{thm}
\declaretheorem[name=Lemma,numberlike=thm]{lem}
\declaretheorem[name=Corollary,numberlike=thm]{cor}
\declaretheorem[name=Proposition,numberlike=thm]{prop}

\declaretheorem[name=Definition,numberlike=thm,style=definition,qed=\(\lozenge\)]{defn}

\declaretheorem[name=Remark,numberlike=thm,style=definition,qed=\(\lozenge\)]{rem}

\newtheorem{thmintro}{Theorem}

\newcommand{\fB}{\mathfrak{B}}
\newcommand{\fD}{\mathfrak{D}}
\newcommand{\fC}{\mathfrak{C}}
\newcommand{\fBH}{\mathfrak{B}(H)}
\newcommand{\sF}{\mathscr{F}}
\newcommand{\RR}{\mathbb{R}}

\newcommand{\olO}{\overline{\Omega}}
\newcommand{\olC}{\overline{C}}
\newcommand{\dO}{\partial\Omega}
\newcommand{\dom}{\mathrm{dom}}
\newcommand{\im}{\mathrm{im}}
\newcommand{\id}{\mathrm{id}}
\newcommand{\Ad}{\mathrm{Ad}}
\newcommand{\secInd}{\mathrm{Ind}_\mathrm{sec}^\Gamma}
\newcommand{\Ind}{\mathrm{Ind}^\Gamma}
\newcommand{\supp}{\mathrm{supp}}

\newcommand{\Dirac}{\slashed{D}}

\newcommand{\gsim}{\gtrsim}
\newcommand{\loc}{\mathrm{loc}}

\numberwithin{equation}{section} 

\author{Matti Lyko\thanks{Institut f{\"u}r Mathematik und Informatik der Universit{\"a}t Greifswald, Walther-Rathenau-Str.\,47, 17489 Greifswald, Germany, 
\href{mailto:matti.lyko@uni-greifswald.de}{matti.lyko@uni-greifswald.de}}
}

\title{A secondary index map on relative $K$-homology}

\begin{document}

\maketitle

\begin{abstract}
  We introduce a secondary index map from relative K-homology to the K-theory of a relative structure algebra. We prove an algebraic Mayer-Vietoris sequence for the relative structure algebra, and prove compatibility of the secondary index map with the Mayer-Vietoris boundary map. As an application, we compute the secondary index of the Dirac operator on an even-dimensional spin manifold with collared metric having uniformly positive scalar curvature on the boundary. This computation leads to a refinement of the delocalized APS-index theorem. 
\end{abstract}

\setcounter{tocdepth}{2}
\tableofcontents

\paragraph{Acknowledgements.}
This article arose out of my dissertation, written at the University of Greifswald under the supervision of Alexander Engel. I want to thank him for his constant support, both of the mathematical and the general variety. I also thank Ulrich Bunke and Paolo Piazza for bringing to my attention the papers \cite{Bunke2018} and \cite{Zenobi2017}, respectively. I furthermore acknowledge financial support by the Deutsche Forschungsgemeinschaft DFG (grant EN 1163/7-1, project number 510244757, Uniformly elliptic boundary value problems).

\section{Introduction}

Coarse index theory provides a powerful tool to study existence and properties of metrics of uniformly positive scalar curvature (upsc) on non-compact spin manifolds \cite{Roe1996}. For $X$ such a manifold, the basic set-up is this: There is an exact sequence
\begin{equation}\label{intro_les}
  K_{*+1}\left( D^*(X)\right) \longrightarrow K_*\left(X\right)  \xlongrightarrow{\mathrm{Ind}} K_*\left(C^*(X)\right) \; ,
\end{equation}
where $K_*(X)$ is the analytic K-homology of $X$, and $C^*(X)$ and $D^*(X)$ denote the \emph{Roe algebra} and \emph{structure algebra} of $X$, respectively, and $\mathrm{Ind}$ is the \emph{coarse index map}. The Dirac operator $\Dirac$ over $X$ defines a class in $K_*(X)$. If $X$ carries a upsc-metric, then the coarse index $\mathrm{Ind}(\Dirac)$ vanishes. By exactness, the K-homology class of $\Dirac$ then lifts to the structure algebra. In fact, there is a canonical lift, called the \emph{higher $\rho$-invariant} $\rho(\Dirac) \in K_{*+1}(D^*(X))$. 

The coarse index map is a powerful obstruction to upsc, and the higher $\rho$-invariant helpful in distinuishing different upsc-metrics on $M$. Even if one is interested in compact manifolds only, coarse index theory still provides value: For a compact manifold with fundamental group $\Gamma$, one can pass to its universal cover. There is an equivariant counterpart of the exact sequence \eqref{intro_les}, which is compatible with the action of $\Gamma$ by deck transformations. In this article, the $\Gamma$-equivariant objects will be indicated by adorning their non-equivariant counterpart with a superscript $(-)^\Gamma$. The equivariant Roe algebra of the universal cover is Morita-equivalent to the (reduced) group $C^*$-algebra $C^*\Gamma$. Generally, the equivariant index is a strictly finer obstruction to positive scalar curvature than the Fredholm index on the base space.

The present article deals specifically with the positive scalar curvature on manifolds with boundary. In order for the existence question for psc-metrics on such manifolds to be interesting, one has to impose boundary conditions on the metric (see e.g. \cite{BaerHanke2023}). The easiest such boundary condition is that the metric have product structure near the boundary. We will call such metrics \emph{collared}. Note that any collared psc-metric induces a psc-metric on the boundary. To detect whether a given (compact) manifold $\olO$ with boundary admits a collared psc-metric, Chang, Weinberger and Yu \cite{Chang2015} introduced a relative Baum-Connes map\footnote{Chang, Weinberger and Yu actually work with Real K-theory. Moreover, they generally consider the maximal relative group $C^*$-algebra, unless the inclusion induces an injection $\pi_1(\dO) \hookrightarrow \pi_1(\olO)$, in which case they also consider the reduced group $C^*$-algebra. For the sake of an efficient introduction, we will gloss over these distinctions.}
$$
  K_*(\olO,\dO) \longrightarrow K_*(C^*(\pi_1(\olO),\pi_1(\dO)))
$$
from the relative K-homology of the pair $(\olO,\dO)$ to the K-theory of a relative group $C^*$-algebra. They show that if $\olO$ admits a collared psc-metric, then the image of the relative K-homology class of the Dirac operator under this map vanishes. Under certain conditions on the fundamental groups of $\olO$ and $\dO$, they also show the converse, i.e. that vanishing of the image is sufficient for the existence of this type of psc-metric (see also \cite{Schick2021}).

Given that any collared psc-metric restricts to a psc-metric on the boundary, it is also natural to ask the relative question: If the boundary has a psc-metric, does this metric extend to a collared psc-metric on the entire manifold? An affirmative answer to this question is obstructed by the so-called \emph{delocalized APS-index} $\Ind_\mathrm{APS}(\Dirac_\Omega)$, which is an element in the equivariant Roe algebra of the universal cover $X$ of $\olO$. The delocalized APS-index theorem states that, upon being pushed to the K-theory of $D^*(X)^\Gamma$, the delocalized APS-index coincides with the higher $\rho$-invariant of the boundary, i.e.
$$
  \Ind_\mathrm{APS}(\Dirac_\Omega) = \rho^\Gamma(\Dirac_{\dO}) \; \in \; K_*\left( D^*(X)^\Gamma \right) \; ,
$$
where the respective inclusion maps into $K_*( D^*(X)^\Gamma)$ are left implicit. The delocalized APS-index theorem was proved by Piazza and Schick \cite{Piazza2014} in the even-dimensional case, and by Xie and Yu \cite{Xie2014} in the odd-dimensional case, in both cases assuming $\olO$ to have finite width. It has as corrolary the psc-bordism invariance of the higher $\rho$-invariant. Indeed, the theorem is part of a more general framework of mapping scalar curvature to analysis \cite{Piazza2014}.

The main contribution of the present article is the definition and study of a secondary index map from equivariant relative K-homology to the K-theory of an equivariant relative structure algebra. Let us describe the general set-up first. Let $X$ be a locally compact metric space carrying a free, proper, isometric action by a countable discrete group $\Gamma$, and let $Z \subseteq X$ a closed, invariant subspace. We consider the \emph{relative structure algebra} $D^*(X,Z)^\Gamma$ generated by all equivariant pseudo-local finite-propagation operators over $X$ that are locally compact over $X\setminus Z$. The algebra $D^*(X,Z)^\Gamma$ has been considered previously, though to the author's knowledge never as primary object of interest.

There is a natural isomorphism between the equivariant relative K-homology $K_*^\Gamma(X,Z)$ and the K-theory $K_{*+1}( D^*(X)^\Gamma/D^*(X,Z)^\Gamma )$. We define the \emph{secondary index map} as the composition
$$
\secInd: \; K_*^\Gamma(X,Z) \cong K_{*+1}\left( \frac{D^*(X)^\Gamma}{D^*(X,Z)^\Gamma} \right) \xlongrightarrow{\partial} K_*(D^*(X,Z)^\Gamma) \; .
$$
The relative structure algebra combines secondary information on $Z$ with coarse information on its complement. This hybrid nature manifests in an algebraic Mayer-Vietoris sequence. The boundary map of this sequence is compatible with the boundary map on K-homology. Specifically, we will prove the following.

\begin{thmintro}\label{IntroMV}
  Let $X$ be a locally compact metric space carrying a free, proper, isometric action by a countable discrete group $\Gamma$, and let $Z \subseteq X$ a closed, invariant subspace. There is an algebraic Mayer-Vietoris sequence
  $$
    K_*(C^*(Z)^\Gamma) \to K_*(C^*(X)^\Gamma) \oplus K_*(D^*(Z)^\Gamma) \to K_*(D^*(X,Z)^\Gamma) \xrightarrow{\delta} K_{*-1}(C^*(Z)^\Gamma) \; .
  $$
  Its boundary map fits into the commutative diagram \\
\centerline{
  \xymatrix{
    K^\Gamma_*(X,Z) \ar[r]^\partial \ar[d]_{\secInd} & K^\Gamma_{*-1}(Z) \ar[d]^{\Ind} \\
    K_*(D^*(X,Z)^\Gamma) \ar[r]_\delta & K_*(C^*(Z)^\Gamma)
  }
}
where $\mathrm{Ind}^\Gamma$ denotes the equivariant coarse index map.
\end{thmintro}

We will apply this machinery to the situation where $X=\olO$ is a spin manifold with boundary $Z=\dO$, which carries a free, proper, isometric action by the group $\Gamma$, which preserves the boundary. We assume the metric on $\olO$ to be collared, and $\dO$ to have uniformly positive scalar curvature. In this setting, we will reframe the delocalized APS-index theorem using the secondary index map. To motivate this, consider the diagram \\
\begin{equation}\label{IntroCommDiag}
  \xymatrix{
 & K^{\Gamma}_{*}(\olO,\dO) \ar[d]_{\secInd}\ar[r]^-\partial & K^{\Gamma}_{*-1}(\dO)\ar[d]^{\Ind} \\
 K_{*}(C^*(\olO)^\Gamma)\oplus K_{*}( D^*(\dO)^\Gamma) \ar[r] & K_{*}( D^*(\olO,\dO)^\Gamma) \ar[r]^-\delta \ar[d] & K_{*-1}( C^*(\dO)^\Gamma) \\
 & K_{*}( D^*(\olO)^\Gamma) &
}
\end{equation}
where the horizontal row comes from the algebraic Mayer-Vietoris sequence, the middle column from the K-theory long exact sequence, and the square commutes due by \cref{IntroMV}. Since we assume $\dO$ to have upsc, we have $\Ind(\Dirac_{\dO}) = 0$. Thus,
$$
  \delta\left(\secInd(\Dirac_\Omega)\right) = \Ind(\partial [\Dirac_\Omega ]) = \Ind(\Dirac_{\dO}) = 0 
$$
via boundary-of-Dirac-is-Dirac. By exactness, $\secInd(\Dirac_\Omega)$ can therefore be written as the sum of an element of $K_{*}(C^*(\olO)^\Gamma)$ and an element of $K_{*}( D^*(\dO)^\Gamma )$. In this geometric set-up, there are natural candidates for these elements, namely $\Ind_\mathrm{APS}(\Dirac_\Omega)$ and $\rho^\Gamma(\Dirac_{\dO})$. In the case that $\olO$ is even-dimensional, we will show that these elements indeed do the trick.

\begin{thmintro}\label{IntroAPS}
  Let $\olO$ be an even-dimensional spin manifold with boundary, which carries a free, proper, isometric action by the discrete group $\Gamma$. Suppose that the metric on $\olO$ is collared, and that the boundary has uniformly positive scalar curvature. Then,
  $$
    \secInd(\Dirac_\Omega) = \Ind_\mathrm{APS}(\Dirac_\Omega) + \rho^\Gamma(\Dirac_{\dO}) \; .
  $$
\end{thmintro}

From this formula one can deduce the delocalized APS-theorem. Indeed, by exactness of the middle column of the diagram \eqref{IntroCommDiag}, $\secInd(\Dirac_\Omega)$ vanishes upon being pushed to $K_*(D^*(\olO)^\Gamma)$. Hence, we deduce the delocalized APS-theorem
$$
  \Ind_\mathrm{APS}(\Dirac_\Omega) = - \rho^\Gamma(\Dirac_{\dO}) \in K_*\left(D^*(\olO)^\Gamma\right) \; .
$$
Our version of theorem differs from that of \cite{Piazza2014} by a sign stemming from a different choice in orientation, see \cref{SignAPS}. We also note that we do not need to assume $\olO$ to have finite width.

We prove \cref{IntroAPS} via reduction to an explicit computation on a cylinder, which makes use of a representation of the relative class $[\Dirac_\Omega] \in K_*^\Gamma(\olO,\dO)$ via boundary conditions. That sufficiently regular boundary conditions give rise to relative K-homology classes has been proved for first-order operators on compact manifolds by Baum, Douglas and Taylor \cite{BDT1989}. Extensions to first-order operators on non-compact manifolds, and to the minimal boundary condition for higher-order operators on (possibly non-compact) manifolds have been given by Forsyth--Goffeng--Messland--Rennie and Fries respectively \cite{Forsyth2019,Fries2025}. Recently, Engel and the author proved the result for general regular boundary conditions for operators of arbitrary order on possibly non-compact manifolds \cite{Engel2026}.

The explicit computation on which our proof relies is the reason why we prove \cref{IntroAPS} only in even dimensions. In odd dimensions we have found the boundary map to be less amenable to direct computation. However, we expect this to be a limitation of our methods rather than the theorem. We note that Piazza and Schick proved the delocalized APS-theorem only in even dimensions for the same reason. The proof of the odd-dimensional case by Xie and Yu uses localization algebras instead. There is an alternative proof of the odd-dimensional case given by Zenobi in \cite{Zenobi2017}, which is closer to our set-up. There, product operations are used to reduce the odd-dimensional case to the even-dimensional case by crossing with a circle. It seems plausible that this approach would also work here, though one would first have to construct the necessary product and prove its compatibility with the secondary index map.

We mention that Zeidler \cite{Zeidler2016} proved a similar strengthening of the delocalized APS-theorem, also using the framework of localization algebras. There, Zeidler equates the difference of the delocalized APS-index and the higher $\rho$-invariant of the boundary in the K-theory of the relative structure algebra with a \emph{$\rho$-invariant of the bordism}. A forteriori, we know this $\rho$-invariant of the bordism to coincide with our secondary index, though Zeidler's construction of this object was non-constructive and relied on the assumption of finite width. Hence, our \cref{IntroAPS} can also be seen as a conceptual complement to Zeidler's result.

In \cite{Bunke2018}, Bunke provides an analogous conceptual framework for the delocalized APS-index theorem and Zeidler's strengthening thereof using coarse homotopy theory. In that approach, a cone sequence replaces exact sequences in operator K-theory, and a strengthened delocalized APS-theorem is proved in that context. That the former gives rise to the latter through application of a suitable homology functor seems plausible, but does not appear to have been worked out yet. 

Lastly, we define a \emph{relative index map} 
$$
  \Ind_\mathrm{rel}: \; K^\Gamma_*(X,Z) \longrightarrow K_*\left( C^*(X,Z)^\Gamma \right) \; .
$$
to the K-theory of the relative Roe algebra as the composition of the secondary index map with a map forgetting the secondary information over $Z$. In return for the discarded information, we gain favorable structural properties in the form of a mapping between long exact sequences. Specifically, we will prove the following.

\begin{thmintro}\label{IntroTransform}
  Let $X$ be a locally compact metric space carrying a free, proper, isometric action by a countable discrete group $\Gamma$, and let $Z \subseteq X$ a closed, invariant subspace. There is a commutative diagram \\
  \centerline{\xymatrix{
    K_*^{\Gamma}(Z) \ar[r] \ar[d]_{\Ind} & K^{\Gamma}_*(X) \ar[r] \ar[d]_{\Ind} & K^{\Gamma}_*(X,Z) \ar[r]^\partial \ar[d]_{\Ind_\mathrm{rel}} & K_{*-1}^{\Gamma}(Z) \ar[d]_{\Ind} \\
    K_*\left(C^*(Z)^\Gamma\right) \ar[r] & K_*\left(C^*(X)^\Gamma\right) \ar[r] & K_*\left(C^*(X,Z)^\Gamma\right) \ar[r]_\partial & K_{*-1}\left(C^*(Z)^\Gamma\right)
  }}
\end{thmintro}

\cref{IntroTransform} and the algebraic Mayer-Vietoris sequence of \cref{IntroMV} can be synthesized into the following diagram: \\
\centerline{\xymatrix@C=10pt{
K_*^{\Gamma}(Z) \ar[r] \ar[d]_{\Ind} & K^{\Gamma}_*(X) \ar[r] \ar@{.>}[d]_{\Ind\oplus 0} & K^{\Gamma}_*(X,Z) \ar[r]^\partial \ar[d]_{\secInd} & K_{*-1}^{\Gamma}(Z) \ar[d]_{\Ind} \\ 
K_*(C^*(Z)^\Gamma) \ar[r] \ar@{=}[d] & K_*(C^*(X)^\Gamma)\oplus K_*(D^*(Z)^\Gamma) \ar@{.>}[d]_{\id\oplus 0} \ar[r] & K_*(D^*(X,Z)^\Gamma) \ar[d] \ar[r]^\delta & K_{*-1}(C^*(Z)^\Gamma)\ar@{=}[d] \\
K_*(C^*(Z)^\Gamma) \ar[r] & K_*(C^*(X)^\Gamma) \ar[r] & K_*(C^*(X,Z)^\Gamma) \ar[r]^\partial & K_{*-1}(C^*(Z)^\Gamma)
}}
Here, the third vertical composition is the relative index map $\Ind_\mathrm{rel}: \, K^\Gamma_*(X,Z) \to K_*( C^*(X,Z)^\Gamma )$. The dotted vertical arrows indicate that \emph{not} all small squares containing them commute due to the presence of the summand $K_*(D^*(Z)^\Gamma)$. Were one to omit that summand the entire diagram would commute, but the second row would no longer be exact. The large squares from the first to the third row do however commute. 

The author would not expect there to be more appropriate exact sequence to fill the middle row. Hence, relative structure algebra and secondary index map likely do not themselves fall into a structurally clean framework of a diagram between long exact sequences. Nevertheless, they provide interesting secondary information that is missed by the structurally favorable relative Roe algebra, especially since the group $K_*(C^*(X,Z)^\Gamma)$ has the unfortunate property of vanishing whenever $Z\hookrightarrow X$ is an equivariant coarse equivalence.

We note that Chang, Weinberger and Yu have a very similar diagram of long exact sequences for their relative Baum-Connes map as appears in \cref{IntroTransform}, only that the lower row involves the fundamental group of the boundary in addition to that of the entire manifold. This provides a richer structure, particularly if one is primarily interested in (covers of) compact manifolds. We suspect that one can upgrade the relative structure algebra and secondary index map to incorporate the fundamental group of the boundary as well, so that the relative Baum-Connes map of Chang, Weinberger and Yu factors through this upgraded secondary index map. However, we will not pursue this here.

The present article arose out of the author's dissertation \cite{LykoThesis}, which was concerned with uniform K-homology and uniformly elliptic operators on manifolds with boundary and bounded geometry. There, the secondary index map and its application to the delocalized APS-index theorem were discussed in the uniform setting. Those parts of the dissertation formed the baseline for this article, as, to the author's knowledge, the material does not appear in the literature even in the standard, non-uniform setting.

\paragraph{Structure of this article.}
In \cref{SecRelStructureAlg} we recall some definitions surrounding Roe algebras and structure algebras, define the relative structure algebra, and prove the algebraic Mayer-Vietoris sequence. In \cref{SecSecInd} we define the secondary index map, prove its compatibility with the algebraic Mayer-Vietoris sequence, and define the relative index map into the K-theory of the relative Roe algebra. In \cref{SectionBVPs} we recall the bare minimum about boundary conditions and their relative K-homology classes, which is needed for the proof of our index formula in \cref{SecAPS}.

\section{The relative structure algebra}\label{SecRelStructureAlg}

Let $X$ be a locally compact proper metric space, and $Z$ a closed subspace. Suppose that $X$ carries a free and proper action of a discrete, countable group $\Gamma$ by isometries, and that $Z$ is invariant under this action. Throughout, we fix a very ample $\Gamma$-equivariant representation $\rho$ of $C_0(X)$ on a separable $\Gamma$-Hilbert space $H$\footnote{This means that $(H,\rho)$ is a countable direct sum of copies of a $\Gamma$-equivariant representation $(H',\rho')$ such that $\rho'(f)$ is only compact if $f=0$.}. The representation $\rho$ extends to a representation of the algebra of bounded Borel-measurable functions on $X$. In particular, we can make sense of the projection $P_Z:=\rho(\mathbf{1}_Z)$. Consider the representation $\rho_Z$ on $H_Z:=P_ZH$ defined by $\rho_Z(g)=\rho(f\mathbf{1}_Z)$. We will assume throughout that this is also a very ample representation of $C_0(Z)$.

Recall that an operator $T\in\fBH$ is said to have \emph{finite propagation} if there is $R\geq 0$ so that $\rho(g)T\rho(f)=0$ whenever the supports of $f$ and $g$ have distance at least $R$. The operator $T$ is said to be \emph{pseudo-local} over $X$ if $[\rho(f),T]$ is compact for all $f\in C_0(X)$. Lastly, $T$ is said to be \emph{locally compact} over an open subset $U\subseteq X$ (which may be $X$ itself) if $T\rho(f)$ and $\rho(f)T$ are compact for all $f\in C_0(U)$.

\begin{defn}
  Let $D^*(X)^\Gamma$ denote the $C^*$-algebra generated by all pseudo-local $\Gamma$-invariant finite-propagation operators. Let $D^*(X,Z)^\Gamma$ denote the $C^*$-algebra generated by all pseudo-local $\Gamma$-invariant finite-propagation operators that are locally compact over $X\setminus Z$.
\end{defn}

The equivariant Roe algebra $C(X)^\Gamma$ is recovered as the special case of $D(X,\emptyset)^\Gamma$, while $D^*(X)^\Gamma$ itself is recovered as $D^*(X,X)^\Gamma$. The algebra $D^*(X,Z)^\Gamma$ forms a closed, two-sided $^*$-ideal in $D^*(X)^\Gamma$, and in turn the Roe algebra $C^*(X)^\Gamma$ is such an ideal in $D^*(X,Z)^\Gamma$. While these algebras depend on the choice of very ample representation, their K-theories are independent of it up to canonical isomorphism. Each of these algebras has a non-equivariant counterpart generated by the respective not-necessarily-equivariant operators. We denote these algebras by ommiting the superscript $(-)^\Gamma$.

Recall that a map $\Phi: \, X \to Y$ of proper metric spaces is called \emph{coarse} if it is proper, and for all $R\geq 0$ there is $S\geq 0$ such that $d_X(x,y)\leq R$ implies $d_Y(\Phi(x),\Phi(y))\leq S$. Moreover, $\Phi$ is called \emph{uniform} if it is both continuous and coarse. If $Z\subseteq X$ and $W\subseteq Y$ are closed subspaces, $\Phi$ is called a \emph{uniform map of pairs} if additionally $\Phi(Z)\subseteq W$. The K-theory groups $K_*(D^*(X)^\Gamma)$ are functorial under $\Gamma$-equivariant uniform maps, see e.g. \cite[~Prop. 4.3.24]{Siegel2012}, and the groups $K_*(D^*(X,Z)^\Gamma)$ are functorial under $\Gamma$-equivariant uniform maps of pairs.

We shall also need versions of these algebras that are localized near a closed subspace. Let $A\subseteq X$ be a closed $\Gamma$-invariant subspace. We say that an operator $T\in\fBH$ is \emph{supported near $A$} if there is $R\geq 0$ so that $T\rho(f)=\rho(f)T=0$ whenever the support of $f$ has disntance at least $R$ from $A$. 

\begin{defn}
  Define $D^*(A\subseteq X,Z)^\Gamma$ as the $C^*$-algebra generated by all pseudo-local $\Gamma$-invariant finite-propagation operators that are supported near $A$ and locally compact over $X\setminus Z$.
\end{defn}

Again, we get the localized Roe algebra $C^*(A\subseteq X)^\Gamma$ as the special case $D^*(A\subseteq X, \emptyset)$, and $D^*(A\subseteq X)$ as $D^*(A\subseteq X,X)$. The following lemma is proved in e.g. \cite[~Prop. 4.3.12]{Siegel2012}.

\begin{lem}\label{RoeIncl}
  The inclusion $A\hookrightarrow X$ induces an isomorphism
  $$
    K_*\left( C^*(A)^\Gamma \right) \xlongrightarrow{\sim} K_*\left( C^*(A\subseteq X)^\Gamma \right) \; .
  $$
\end{lem}
  
We denote by $\fD(X)^\Gamma$ the $C^*$-algebra of all equivariant pseudo-local operators over $X$, by $\fC(X)^\Gamma$ that of all equivariant locally compact operators, and by $\fD(X,Z)^\Gamma$ the closed two-sided $^*$-ideal in $\fD(X)^\Gamma$ containing all operators that are additionally locally compact over $X\setminus Z$. We will denote the non-equivariant counterparts of these algebras by omitting the superscript $(-)^\Gamma$. The following lemma can be proved analogously to the non-equivariant setting, see e.g. \cite[~Lemma 1.6]{Higson1995}.

\begin{lem}\label{TrivialKtheory}
  $K_*(\fC(X)^\Gamma)=0$\; .
\end{lem}

Fix a locally finite open cover $\{U_i\}_{i\in I}$ of $X$ with uniformly bounded diameters and a subordinate partition of unity $\{\psi_i\}_{i\in I}$ so that 
$$
  \forall \, \gamma \in \Gamma, \; i\in I \; \exists \, j \in J: \gamma U_i = U_j \; , \; \gamma \psi_i = \psi_j \; .
$$
Such data can be obtained for example by pulling back the corresponding data from the quotient $X_\Gamma$ of $X$ by $\Gamma$. There is a truncation map
\begin{equation}\label{TruncMap}
  \mathrm{trunc}: \; \fD(X)^\Gamma \longrightarrow D^*(X)^\Gamma \; , \; T \longmapsto \sum_i \sqrt{\psi_i} T \sqrt{\psi_i} \; .
\end{equation}
It has the property that $T$ and $\mathrm{trunc}(T)$ differ by an equivariant locally compact operator, and it preserves the properties of being locally compact over an open subset, and being supported near a closed one.

\begin{lem}\label{StructureDual} 
  Suppose that $Z\subseteq A$. Then,
  $$
  \frac{D^*(A \subseteq X, Z)^\Gamma}{C^*(A \subseteq X)^\Gamma} = \frac{\fD(X, Z)^\Gamma}{\fC(X)^\Gamma}\; .
  $$
\end{lem}

\begin{proof}
  The properties of the truncation map \eqref{TruncMap} together with use of a cut-off function yield that $\fD(X, Z)^\Gamma = D^*(A \subseteq X, Z)^\Gamma + \fC(X)^\Gamma$, giving one inclusion. The observation that $C^*(A\subseteq X)^\Gamma$ equals the intersection of $D^*(A \subseteq X, Z)^\Gamma$ and $\fC(X)^\Gamma$ gives the reverse inclusion. 
\end{proof}

Using an analogous truncation argument, one also shows the following.

\begin{lem}\label{RelQuot}
  It holds that
  $$
  \frac{D^*(X)^\Gamma}{D^*(X,Z)^\Gamma} = \frac{\fD(X)^\Gamma}{\fD(X,Z)^\Gamma}\; .
  $$
\end{lem}

We denote by $X_\Gamma$ the quotient of $X$ by $\Gamma$, and similarly $Z_\Gamma$ and $A_\Gamma$. Note that since $Z$ and $A$ are assumed $\Gamma$-invariant, the quotients $Z_\Gamma$ and $A_\Gamma$ coincide with the image of the respective subspace under the quotient map $X\to X_\Gamma$.

\begin{lem}\label{EquivQuotient}
  There is an isomorphism 
  $$
  \frac{D^*(A\subseteq X,Z)^\Gamma}{C^*(A\subseteq X)^\Gamma} \cong \frac{D^*(A_\Gamma\subseteq X_\Gamma, Z_\Gamma)}{C^*(A_\Gamma \subseteq X_\Gamma)} \; .
  $$
\end{lem}

\begin{proof}
  Pick a locally finite cover of $X_\Gamma$ by open sets $U_i$ with uniformly bounded diameters such that the pre-image $\tilde{U}_i$ of $U_i$ in $X$ is isomorphic to $U_i \times \Gamma$. Any operator over $X_\Gamma$ with support in $U_i\times U_i$ can then be lifted to a $\Gamma$-equivariant operator over $X$ with support in $\tilde{U}_i \times \tilde{U}_i$. Conversely, any such operator over $X$ arises from a unique operator over $X_\Gamma$ in this way. This procedure preserves the properties of being pseudo-local, being supported near $A$ or $A_\Gamma$, and being locally compact on the complement of $Z$ or $Z_\Gamma$. Finally, the truncation map \eqref{TruncMap} can be used to represent a class in either quotient by the strong limit of a series of operators with supports in some $U_i$ or $\tilde{U}_i$, respectively. For the details we refer to the discussion around Theorem 4.3.25 in \cite{Siegel2012}.
\end{proof}

\begin{prop} \label{RelLocalized}
  Suppose that $Z\subseteq A$. Then, the inclusion $A\hookrightarrow X$ induces an isomorphism
  $$
  K_*\left( D^*(A,Z)^\Gamma \right) \xlongrightarrow{\sim} K_*\left( D^*(A\subseteq X,Z)^\Gamma \right) \; .
  $$
\end{prop}

\begin{proof}
  The proof is an essentially verbatim adaptation of the proof of \cite[Prop.~4.3.34]{Siegel2012}. The extension by zero of operators over $A$ to operators over $X$ gives rise to the diagram \\
  \centerline{\xymatrix{
    0 \ar[r] & C^*(A)^\Gamma \ar[r] \ar[d] & D^*(A,Z)^\Gamma \ar[r] \ar[d] & \frac{D^*(A,Z)^\Gamma}{C^*(A)^\Gamma} \ar[r] \ar[d] & 0 \\
    0 \ar[r] & C^*(A\subseteq X)^\Gamma \ar[r] & D^*(A\subseteq X,Z)^\Gamma \ar[r] & \frac{D^*(A\subseteq X,Z)^\Gamma}{C^*(A\subseteq X)^\Gamma} \ar[r] & 0
  }}
  Once we show that the first and third vertical arrow are K-theory isomorphisms, the claim will follow from the Five lemma applied to the associated diagram between K-theory long exact sequences. The first vertical arrow is a K-theory isomorphism by \cref{RoeIncl}. To see that third vertical arrow is a K-theory isomorphism as well, note that it fits into the diagram\\
  \centerline{\xymatrix{
    D^*(A,Z)^\Gamma/C^*(A)^\Gamma \ar[r] \ar[d] & \fD(A,Z)^\Gamma/\fC(A)^\Gamma \ar[d] \\
    D^*(A\subseteq X,Z)^\Gamma/C^*(A\subseteq X)^\Gamma \ar[r] & \fD(X,Z)^\Gamma/\fC(X)^\Gamma 
  }}
  The horizontal arrows are isomorphisms by \cref{StructureDual}. The right-vertical arrow is a K-theory isomorphism as well. Indeed, through a combination of \cref{StructureDual} and \cref{EquivQuotient} it suffices to consider the non-equivariant situation. Via \cref{TrivialKtheory} it further suffices that $\fD(A,Z)\hookrightarrow \fD(X,Z)$ induces a K-theory isomorphism, which in turn follows from the fact that both algebras have K-theory naturally isomorphic to $\fD(Z)$ \cite[Lemma~2.2]{Higson1995}. We conclude that $D^*(A,Z)^\Gamma/C^*(A)^\Gamma \to D^*(A\subseteq X,Z)^\Gamma/C^*(A\subseteq X)^\Gamma$ is a K-theory isomorphism as well, finishing the proof. 
\end{proof}

For the case $A=Z$, \cref{RelLocalized} specializes to give the following.

\begin{cor}\label{RelLocalizedAisZ}
  The inclusion $Z\hookrightarrow X$ induces an isomorphism
  $$
  K_*\left( D^*(Z)^\Gamma \right) \xlongrightarrow{\sim} K_*\left( D^*(Z\subseteq X,Z)^\Gamma \right) \; .
  $$
\end{cor}

In the case that $Z\subseteq X$ is $c$-dense for some $c$, i.e. $N_c(Z)=X$, this allows us to identify the K-theory of the relative structure algebra.

\begin{cor}\label{IsoRelBdry}
  Suppose that $N_c(Z)=X$ for some $c\geq 0$. Then, the inclusion $Z\hookrightarrow X$ induces an isomorphism
  $$
  K_*\left( D^*(Z)^\Gamma \right) \xlongrightarrow{\sim} K_*\left( D^*(X,Z)^\Gamma \right) \; .
  $$
\end{cor}

More generally, there is an algebraic Mayer-Vietoris sequence expressing the K-theory of the relative structure algebra of $(X,Z)$ in terms the uniform structure algebra of $Z$, as well as the Roe algebras of $X$ and $Z$.

\begin{thm}\label{AlgebraicMV}
There is an algebraic Mayer-Vietoris sequence\\
\centerline{\xymatrix{
K_0\left(C^*(Z)^\Gamma\right) \ar[r] & K_0\left(C^*(X)^\Gamma\right)\oplus K_0\left(D^*(Z)^\Gamma\right) \ar[r] & K_0\left(D^*(X,Z)^\Gamma\right) \ar[d]^\delta \\
K_1\left(D^*(X,Z)^\Gamma\right) \ar[u]^\delta & K_1\left(C^*(X)^\Gamma\right)\oplus K_1\left(D^*(Z)^\Gamma\right) \ar[l] &  K_1\left(C^*(Z)^\Gamma\right) \ar[l]
}}
\end{thm}

\begin{proof}
  Using cutoff functions one shows that the algebra $D^*(X,Z)^\Gamma$ can be written as $D^*(X,Z)^\Gamma=C^*(X)^\Gamma + D^*(Z\subseteq X,Z)^\Gamma$. Both summands are closed two-sided $^*$-ideals, and their intersection is given by $C^*(Z\subseteq X)^\Gamma$. Application of \cref{RoeIncl} and \cref{RelLocalized} to the corresponding Mayer-Vietoris sequence on K-theory concludes the proof.
\end{proof}

\begin{rem}\label{CoarseEqMV}
  The algebraic Mayer-Vietoris sequence allows us to extend the conclusion of \cref{IsoRelBdry} to the case where $Z$ is coarsely equivalent to $X$, but not necessarily $c$-dense. Indeed, if $Z\hookrightarrow X$ is a coarse equivalence of $\Gamma$-spaces, the map $K_*(C^*(Z)^\Gamma) \to K_*(C^*(X)^\Gamma)$ is an isomorphism. From the Mayer-Vietoris sequence one may then deduce that $K_*(D^*(Z)^\Gamma)\to K_*(D^*(X,Z)^\Gamma)$ is an isomorphism, and that the boundary map $\delta: \,K_*(D^*(X,Z)^\Gamma) \to K_{*-1}(C^*(Z)^\Gamma)$ is zero. The latter fact is a version of bordism invariance of the index, see \cref{BordInvIndex} below.
\end{rem}

\section{The secondary index map}\label{SecSecInd}

As before, let $X$ be a locally compact proper metric space, $Z$ a closed subspace, and suppose that $X$ carries a free, proper, and isometric action of a discrete, countable group $\Gamma$ which leaves $Z$ invariant. Again, we let $X_\Gamma$ and $Z_\Gamma$ denote the corresponding quotients.

Let us briefly recall the definition of the relative equivariant K-homology. For simplicity we focus on degrees zero and one. An \emph{ungraded relative equivariant Fredholm module} over $(X,Z)$ is a triple $(H,\rho,F)$ consisting of a $\Gamma$-Hilbert space, an equivariant representation $\rho: \, C_0(X)\to \fBH$, and an equivariant operator $F\in\fBH$ that is pseudo-local over $X$ so that $I-F^2$ and $F-F^*$ are locally compact over $X\setminus Z$. In other words, $F$ is in $\fD(X)^\Gamma$, and is a self-adjoint involution modulo $\fD(X,Z)^\Gamma$. A \emph{graded relative equivariant Fredholm module} is defined analogously, only that it is additionally assumed that $H$ carries a grading compatible with the $\Gamma$-action, that $\rho$ is a representation by even operators, and that $F$ is an odd operator. 

Then, $K_0^\Gamma(X,Z)$ resp. $K_1^\Gamma(X,Z)$ are defined as the groups generated by graded resp. ungraded relative equivariant Fredholm modules under direct sums, modulo the relations of equivariant unitary equivalence and operator-homotopy through relative equivariant Fredholm modules. A brief introduction to relative equivariant K-homology can be found in \cite[~Appendix A]{Forsyth2019}. Generally, definitions and statements for relative equivariant K-homology are equivalents of their non-equivariant counterparts.

\begin{prop}\label{EquivKhom}
  \begin{itemize}
    \item[(i)] There is a isomorphism
      $$
        K_*^\Gamma(X,Z) \cong K_{*+1}\left( \frac{D^*(X)^\Gamma}{D^*(X,Z)^\Gamma}\right) \; ,
      $$
      which is natural with respect to uniform maps of pairs.
    \item[(ii)] There is an isomorphism
      $$
        K_*^\Gamma(X,Z) \cong K_*(X_\Gamma,Z_\Gamma)
      $$
     which is compatible with boundary maps, i.e. the square \\
      \centerline{\xymatrix{
        K^\Gamma_*(X,Z) \ar[r]^\partial \ar[d]_\sim & K^\Gamma_{*-1}(Z) \ar[d]^\sim \\
        K_*(X_\Gamma) \ar[r]_\partial & K_{*-1}(Z_\Gamma)
      }}
      commutes.
  \end{itemize}
\end{prop}

\begin{proof}
  Ad (i): It holds that 
  $$
    K^\Gamma_*(X,Z) \cong K_{*+1}\left( \frac{\fD(X)^\Gamma}{\fD(X,Z)^\Gamma} \right) = K_{*+1}\left( \frac{D^*(X)^\Gamma}{D^*(X,Z)^\Gamma} \right) \; ,
  $$
  where the first (natural) isomorphism is due to Paschke duality (see e.g. \cite{Higson1995} for the non-equivariant case), and the second equality is due to \cref{RelQuot}.
  
  Ad (ii): Combine \cref{StructureDual} and \cref{EquivQuotient} to see that 
  $$
    \frac{\fD(X,Z)^\Gamma}{\fC(X)^\Gamma} = \frac{D^*(X,Z)^\Gamma}{C^*(X)^\Gamma} \cong \frac{D^*(X_\Gamma,Z_\Gamma)}{C^*(X_\Gamma)} = \frac{\fD(X_\Gamma,Z_\Gamma)}{\fC(X_\Gamma)} \; .
  $$
  Analogously, one finds that $\fD(X)^\Gamma/\fC(X)^\Gamma \cong \fD(X_\Gamma)/\fC(X_\Gamma)$. These isomorphisms fit into the larger diagram \\
  \centerline{\xymatrix{
    0 \ar[r] &  \fD(X,Z)^\Gamma \ar[r] \ar[d] & \fD(X)^\Gamma \ar[r] \ar[d] & \frac{\fD(X)^\Gamma}{\fD(X,Z)^\Gamma} \ar[r] \ar[d] & 0 \\
    0 \ar[r] &  \frac{\fD(X,Z)^\Gamma}{\fC(X)^\Gamma} \ar[r] \ar[d]_\sim & \frac{\fD(X)^\Gamma}{\fC(X)^\Gamma} \ar[r]\ar[d]_\sim & Q(X,Z)^\Gamma \ar[r] \ar[d] & 0 \\
    0 \ar[r] &  \frac{\fD(X_\Gamma,Z_\Gamma)}{\fC(X_\Gamma)} \ar[r] & \frac{\fD(X_\Gamma)}{\fC(X_\Gamma)} \ar[r] & Q(X_\Gamma,Z_\Gamma) \ar[r] & 0 \\
    0 \ar[r] &  \fD(X_\Gamma,Z_\Gamma) \ar[r] \ar[u] & \fD(X_\Gamma) \ar[r] \ar[u] & \frac{\fD(X_\Gamma)}{\fD(X_\Gamma,Z_\Gamma)}\ar[r] \ar[u] & 0
  }}
  where $Q(X,Z)^\Gamma$ and $Q(X_\Gamma,Z_\Gamma)$ denote the appropriate quotient algebras. Commutativity of the middle squares holds because all middle arrows are defined via the same quotienting/lifting procedure described in the proof of \cref{EquivQuotient}. By \cref{TrivialKtheory} the upper- and lower-left, and upper- and lower-middle arrows are K-theory isomorphisms. Hence, upon passing to the corresponding diagram between long exact K-theory sequences, we obtain isomorphisms
  $$
    K_*(\fD(X,Z)^\Gamma) \cong K_*(\fD(X_\Gamma,Z_\Gamma)) \quad , \quad K_*(\fD(X)^\Gamma) \cong K_*(\fD(X_\Gamma))
  $$
  The Five lemma lets us conclude
  $$
    K_*^\Gamma(X,Z) \cong K_{*+1}\left( \frac{\fD(X)^\Gamma}{\fD(X,Z)^\Gamma} \right) \cong K_{*+1}\left( \frac{\fD(X_\Gamma)}{\fD(X_\Gamma,Z_\Gamma)} \right) \cong K_*(X_\Gamma,Z_\Gamma) \; .
  $$
  Compatibility with the boundary maps follows from the fact that the K-theory boundary map is natural with respect to maps of short exact seuquences.
\end{proof}

We come to the definition of the secondary index map.

\begin{defn}
  The \emph{secondary index map} is defined as the composition
  $$
  \secInd: \; K_*^\Gamma(X,Z) \cong K_{*+1}\left( \frac{D^*(X)^\Gamma}{D^*(X,Z)^\Gamma} \right) \xlongrightarrow{\partial} K_*(D(X,Z)^\Gamma) \; .
  $$
\end{defn}

\begin{rem}
  As a consequence of \cref{EquivKhom} (ii), one could in principle forego equivariant K-homology, and work with the pair $(X_\Gamma, Z_\Gamma)$ together with the datum of a locally isometric $\Gamma$-covering determining $(X,Z)$: Suppose for the sake of concreteness that $\olO$ is a complete riemannian manifold with boundary. Associated to any continuous map $f: \, \olO \to B\Gamma$ to the classifying space of $\Gamma$, there is an associated $\Gamma$-covering $X\to \olO$. Equip the total space $X$ with the pullback metric, and the $\Gamma$-action by deck transformations. Let $Z$ denote the pre-image of $\dO$ in $X$. Then $X_\Gamma = \olO$ and $Z_\Gamma = \dO$. In this situation, one may consider the secondary index map
  $$
    \secInd(f): \; K_*(\olO,\dO) \cong K_*^\Gamma(X,Z) \xlongrightarrow{\secInd} K_*(D(X,Z)^\Gamma)
  $$
  with domain the (non-equivariant) K-homology of $(\olO,\dO)$, which takes the classifying map $f$ as an additional datum. This approach has its own advantages, especially if it is the base space $\olO$ which is of primary geometric interest. Nevertheless, for our purposes we find it cleanest to work on the total space directly.
\end{rem}

\begin{prop}\label{RelIndexNatural}
Let $(X,Z)$ and $(Y,W)$ be pairs of proper locally compact metric spaces with free, proper, isometric $\Gamma$-actions fixing the respective subspace, and let $\Phi: \, (X,Z) \to (Y,W)$ be an equivariant uniform map of pairs. Then, the diagram \\
\centerline{\xymatrix{
K^\Gamma_*(X,Z) \ar[d]_{\secInd} \ar[r]^{\Phi_*} &  \ar[d]^{\secInd} K^\Gamma_*(Y,W) \\
K_*(D^*(X,Z)^\Gamma) \ar[r]_{\Phi_*} & K_*(D^*(Y,W)^\Gamma)
}}
commutes. 
\end{prop}

\begin{rem}
  In the case that $\Gamma$ acts non-cocompactly on $X$, one also sees (relative) equivariant K-homology $K^\Gamma_*(X,Z)$ to be defined as the direct limit of $K_*^\Gamma(Y,Y\cap Z)$, taken over all closed, invariant $Y\subseteq X$ on which $\Gamma$ acts cocompactly. This version can be seen as equivariant K-homology with cocompact supports, whereas the version used in this paper can be seen as equivariant locally finite K-homology. There are natural maps
  $$
    \varinjlim K_*^\Gamma(Y, Y\cap Z) \longrightarrow K_*^\Gamma(X,Z) \; , \; \varinjlim K_*\left(D^*(Y, Y\cap Z)^\Gamma\right) \longrightarrow K_*\left(D^*(X,Z)^\Gamma\right)
  $$
  induced by the inclusions $(Y,Y\cap Z)\hookrightarrow (X,Z)$. As a consequence of \cref{RelIndexNatural}, the diagram \\
  \centerline{\xymatrix{
    \varinjlim K_*^\Gamma(Y, Y\cap Z) \ar[r] \ar[d]_{\varinjlim\secInd} & K_*^\Gamma(X,Z) \ar[d]^{\secInd} \\
    \varinjlim K_*\left(D^*(Y, Y\cap Z)^\Gamma\right) \ar[r] &  K_*\left(D^*(X,Z)^\Gamma\right)
  }}
  commutes. Hence, there is an analogous secondary index map on cocompactly supported K-homology, and one may pass frictionlessly from it to its locally finite counterpart.  
\end{rem}

We come to the main theorem of this section.

\begin{thm}\label{IndexCommAlgMV}
Let $X$ be a  proper locally compact metric space with free, proper, isometric $\Gamma$-actions, and $Z$ a $\Gamma$-invariant subspace. Then, the diagram \\
\centerline{\xymatrix{
K^\Gamma_{*}(X,Z) \ar[r]^\partial \ar[d]_{\secInd} & K^\Gamma_{*-1}(Z) \ar[d]^-{\Ind} \\
K_{*}\left( D^*(X,Z)^\Gamma\right) \ar[r]_-\delta & K_{*-1}\left( C^*(Z)^\Gamma \right)
}} 
commutes.
\end{thm}
\begin{proof}
We implicitly use a very ample representation $(H,\sigma)$ such that $(H_Z:=P_ZH, \sigma_Z:=P_Z\sigma(\mathbf{1}_Z\cdot))$ is a very ample representation of $Z$, where $P_Z:=\sigma(\mathbf{1}_Z)$. Then, the upper-corner inclusion of $\fD(Z)^\Gamma \subseteq \fB(H_Z)$ into $\fD(X,Z)^\Gamma \subseteq \fB(H=H_Z\oplus H^\perp_Z)$ realizes the isomorphism $K_*(\fD(Z)^\Gamma)\cong K_*(\fD(X,Z)^\Gamma)$. Restricting the upper corner inclusion gives rise to inclusions $D^*(Z)^\Gamma\hookrightarrow D^*(Z\subseteq X,Z)^\Gamma$ and $C^*(Z)^\Gamma\hookrightarrow C^*(Z\subseteq X)^\Gamma$, with the latter inducing an isomorphism on K-theory by \cref{RoeIncl}. Thus, there is also a well-defined map 
\begin{equation}\label{TauQuot}
\frac{D^*(Z)^\Gamma}{C^*(Z)^\Gamma} \longrightarrow \frac{D^*(Z\subseteq X,Z)^\Gamma}{C^*(Z\subseteq X)^\Gamma} \; .
\end{equation}
For the sake of efficient notation we will denote this and all other maps given or induced by the upper-corner inclusion by $\tau$.

Observe that by \cref{RelQuot} and naturality of the boundary map we have a commutative diagram\\
\centerline{\xymatrix{
K_{*+1}\left( \frac{D^*(X)^\Gamma}{D^{*}(X,Z)^\Gamma} \right) \ar@{=}[r] \ar[d]_{\secInd} & K_{*+1}\left( \frac{\fD(X)^\Gamma}{\fD(X,Z)^\Gamma} \right) \ar[d]^\partial \\
K_{*}\left( D^*(X,Z)^\Gamma\right) \ar[r] & K_{*}\left( \fD(X,Z)^\Gamma \right)
}}
where the right-vertical arrow represents the boundary map on K-homology. It therefore suffices to prove that that the diagram\\
\begin{equation}\label{DiagramMV}
\xymatrix{
K_{*}\left( D^*(X,Z)^\Gamma\right) \ar[rr] \ar[dr]_{\delta} & & K_{*}\left( \fD(X,Z)^\Gamma \right) \ar[dl]^{\Ind} \\
& K_{*-1}\left(C^*(Z)^\Gamma\right) & }
\end{equation}
commutes. Let us unravel the definitions of $\delta$ and $\Ind$. The former is given by the composition\footnote{This description of the Mayer-Vietoris boundary map can be obtained by viewing the Mayer-Vietoris sequence of a decomposition $A=I+J$ as the long exact sequence of the pair $(I\cap J, I \oplus J)$, where $I\cap J$ is included diagonally in $I\oplus J$.}
\begin{align} \label{DescriptionMVboundary}
K_{*}\left( D^*(X,Z)^\Gamma \right) &\rightarrow K_{*}\left( D^*(X,Z)^\Gamma/C^*(X)^\Gamma\right) \\
&= K_{*}\left( D^*(Z\subseteq X,Z)^\Gamma/C^*(Z\subseteq X)^\Gamma\right) \nonumber \\
& \xrightarrow{\partial} K_{*-1}\left(C^*(Z\subseteq X)^\Gamma\right) \nonumber \\
&\xrightarrow{\sim} K_{*-1}\left( C^*(Z)^\Gamma \right) \; . \nonumber
\end{align}
Here, we used that
\begin{equation}\label{eq_cutoff_Z}
  \frac{D^*(Z\subseteq X,Z)^\Gamma}{C^*(Z\subseteq X)^\Gamma} = \frac{D^*(X,Z)^\Gamma}{C^*(X)^\Gamma} \; ,
\end{equation}
which can be obtained by cutting off near $Z$. The coarse index map $\Ind$ is given by 
\begin{align*}
K_{*}\left( \fD(X,Z)^\Gamma \right) &\xrightarrow{\sim} K_{*}\left( \fD(Z)^\Gamma \right) \\
&\xrightarrow{\sim} K_{*}\left( \fD(Z)^\Gamma/\fC(Z)^\Gamma\right) \\
&= K_{*}\left( D^*(Z)^\Gamma/C^*(Z)^\Gamma\right) \\
& \xrightarrow{\partial} K_{*-1}\left(C^*(Z)^\Gamma\right) \; .
\end{align*}
As discussed above the first arrow is inverse to the upper-corner inclusion. Recall also that the upper-corner inclusion induced the map \eqref{TauQuot} on quotients. Combining this with the definitions of $\delta$ and $\Ind$ we write diagram \eqref{DiagramMV} as 
\begin{equation}\label{DiagramMV2}
\xymatrix{
K_{*}\left( D^*(X,Z)^\Gamma\right) \ar[r] \ar[d] & K_{*}\left( \fD(X,Z)^\Gamma \right) \ar[d]^{\sim} \\
K_{*}\left( \frac{D^*(X,Z)^\Gamma}{C^*(X)^\Gamma}\right) \ar@{=}[d] & \ar[l]_-{\tau_*} K_{*}\left( \frac{D^*(Z)^\Gamma}{C^*(Z)^\Gamma}\right) \ar@{=}[d] \\
K_{*}\left( \frac{D^*(Z\subseteq X,Z)^\Gamma}{C^*(Z\subseteq X)^\Gamma}\right) \ar[d]_\partial & \ar[l]_-{\tau_*} K_{*}\left( \frac{D^*(Z)^\Gamma}{C^*(Z)^\Gamma}\right) \ar[d]^\partial \\
K_{*-1}\left( C^*(Z\subseteq X)^\Gamma \right) & \ar[l]_\sim K_{*-1}\left( C^*(Z)^\Gamma \right) }
\end{equation}
where the bottom horizontal isomorphism is also induced by the upper-corner inclusion. The lower square is commutative by naturality of the boundary map. Therefore it suffices to show commutativity of the upper square. 

To that end, note first that by \cref{StructureDual} the quotient in \eqref{eq_cutoff_Z} is also equal to $\fD(X,Z)^\Gamma/\fC(X)^\Gamma$. The diagram\\
\centerline{\xymatrix{
K_{*}\left( \fD(Z)^\Gamma \right) \ar[r]^-{\tau_*} \ar[d] & K_{*}\left( \fD(X,Z)^\Gamma \right) \ar[d] \\
K_{*}\left( \frac{\fD(Z)^\Gamma}{\fC(Z)^\Gamma}\right)\ar[r]_-{\tau_*} & K_{*}\left( \frac{\fD(X,Z)^\Gamma}{  \fC(X)^\Gamma}\right)
}}
commutes, where we recall that $\tau$ always denotes maps induced by upper-corner inclusion. Recall also that the upper-right isomorphism in \eqref{DiagramMV2} is given by 
$$
K_{*}\left( \fD(X,Z)^\Gamma \right) \xrightarrow{\tau_*^{-1}} K_{*}\left( \fD(Z)^\Gamma \right) \longrightarrow K_{*}\left( \frac{\fD(Z)^\Gamma}{\fC(Z)^\Gamma}\right) = K_{*}\left( \frac{D^*(Z)^\Gamma}{C^*(Z)^\Gamma}\right) \; .
$$
The combination of these facts reduces the upper square in \eqref{DiagramMV2} to the commutativity of \\
\centerline{\xymatrix{
K_{*}\left( D^*(X,Z)^\Gamma\right) \ar[r] \ar[d] & K_{*}\left( \fD(X,Z)^\Gamma \right) \ar[d] \\
 K_{*}\left( \frac{D^*(X,Z)^\Gamma}{C^*(X)^\Gamma}\right) \ar@{=}[r]& K_{*}\left( \frac{\fD(X,Z)^\Gamma}{\fC(X)^\Gamma} \right)
}}
which holds even before passing to K-theory. This concludes the proof.
\end{proof}

While our primary interest is the secondary index map, we also describe how it can be used to construct a relative index map to the relative Roe algebra. To that end, define $C^*(X,Z)^\Gamma$ as the quotient $C^*(X)/C^*(Z\subseteq X)^\Gamma$.

\begin{lem}
The inclusion $C^*(X)^\Gamma\hookrightarrow D^*(X,Z)^\Gamma$ induces an isomorphism
$$
\frac{C^*(X)^\Gamma}{C^*(Z\subseteq X)^\Gamma} \xlongrightarrow{\sim} \frac{D^*(X,Z)^\Gamma}{D^*(Z\subseteq X,Z)^\Gamma} \; .
$$
Thus, the composition
$$
D^*(X,Z)^\Gamma \longrightarrow \frac{D^*(X,Z)^\Gamma}{D^*(Z\subseteq X,Z)^\Gamma} \xlongleftarrow{\sim} \frac{C^*(X)^\Gamma}{C^*(Z\subseteq X)^\Gamma}
$$
produces a natural forgetful\footnote{We use the adjective \emph{forgetful} here to indicate that the map $\sF$ is defined by forgetting the information near $Z$.} map
$$
D^*(X,Z)^\Gamma \xlongrightarrow{\sF} C^*(X,Z)^\Gamma \; .
$$
\end{lem}

\begin{proof}
  It holds that $C^*(Z\subseteq X)^\Gamma = D^*(Z\subseteq X,Z)^\Gamma \cap D^*(X,Z)^\Gamma$, and cut-off argument proves $C^*(X)^\Gamma + D^*(Z \subseteq X,Z)^\Gamma = D^*(X,Z)^\Gamma$. 
\end{proof}

Hence, we may define the \emph{relative equivariant index map} as the composition
$$
  \Ind_\mathrm{rel} : \; K^\Gamma_*(X,Z) \xlongrightarrow{\secInd} K_*\left( D^*(X,Z)^\Gamma \right) \xlongrightarrow{\sF_*} K_*\left( C^*(X,Z)^\Gamma \right) \; .
$$
This index map forgets the secondary information near the boundary, and only retains the coarse information in the interior. As compensation, one obtains a transformation between long exact sequences.

\begin{thm}\label{LEStransform}
  There is a commutative diagram \\
\centerline{\xymatrix{
K_*^{\Gamma}(Z) \ar[r] \ar[d]_{\Ind} & K^{\Gamma}_*(X) \ar[r] \ar[d]_{\Ind} & K^{\Gamma}_*(X,Z) \ar[r]^\partial \ar[d]_{\Ind_\mathrm{rel}} & K_{*-1}^{\Gamma}(Z) \ar[d]_{\Ind} \\
K_*\left(C^*(Z)^\Gamma\right) \ar[r] & K_*\left(C^*(X)^\Gamma\right) \ar[r] & K_*\left(C^*(X,Z)^\Gamma\right) \ar[r]_\partial & K_{*-1}\left(C^*(Z)^\Gamma\right)
}}
\end{thm}

\begin{proof}
Commutativity of the first square is due to the naturality of the (non-relative) index map. Using \cref{EquivKhom} and naturality of the K-theory boundary map, it is not hard to see that the relative index map $\Ind_\mathrm{rel}$ can also be written as the composition
$$
  K_{*+1}\left(\frac{D^*(X)^\Gamma}{D^*(X,Z)^\Gamma}\right) \longrightarrow K_{*+1}\left( \frac{D^*(X)^\Gamma/D^*(X,Z)^\Gamma}{C^*(X,Z)^\Gamma} \right) \xlongrightarrow{\partial} K_*\left( C^*(X,Z)^\Gamma \right) \; ,
$$
where one uses that $C^*(X,Z)^\Gamma$ is included as an ideal in $D^*(X)^\Gamma/D^*(X,Z)^\Gamma$. The second square can thus be written as \\
\centerline{\xymatrix{
K_{*+1}\left(\frac{D^*(X)^\Gamma}{C^*(X)^\Gamma}\right) \ar[r] \ar[dd]_{\partial} & K_{*+1}\left(\frac{D^*(X)^\Gamma}{D^*(X,Z)^\Gamma}\right) \ar[d] \\
& K_{*+1}\left( \frac{D^*(X)^\Gamma/D^*(X,Z)^\Gamma}{C^*(X,Z)^\Gamma} \right) \ar[d]^\partial \\
K_*\left(C^*(X)^\Gamma\right) \ar[r] & K_*\left(C^*(X,Z)^\Gamma\right)
}}
which commutes by naturality of the boundary map. Lastly, inserting the definition of the relative index map, write the third square as \\
\centerline{\xymatrix{
K^{\Gamma}_*(X,Z) \ar[rr]^\partial \ar[d]_{\secInd} && K_{*-1}^{\Gamma}(Z) \ar[d]^{\Ind} \\
K_*\left(D^*(X,Z)^\Gamma\right) \ar[r]_{\sF_*} & K_*\left(C^*(X,Z)^\Gamma\right) \ar[r]_\partial & K_{*-1}\left(C^*(Z)^\Gamma\right)
}}
Comparing with \eqref{DescriptionMVboundary}, one sees the composition 
$$
K_*\left(D^*(X,Z)^\Gamma\right) \xrightarrow{\sF_*} K_*\left(C^*(X,Z)^\Gamma\right) \xrightarrow{\partial} K_{*-1}\left(C^*(Z)^\Gamma\right)
$$ 
to coincide with the algebraic Mayer-Vietoris boundary map. Thus, commutativity of the last square is a consequence of \cref{IndexCommAlgMV}.
\end{proof}

As described in the introduction, \cref{LEStransform} and the algebraic Mayer-Vietoris sequence of \cref{AlgebraicMV} can be combined into the diagram: \\
\centerline{\xymatrix@C=10pt{
K_*^{\Gamma}(Z) \ar[r] \ar[d]_{\Ind} & K^{\Gamma}_*(X) \ar[r] \ar@{.>}[d]_{\Ind\oplus 0} & K^{\Gamma}_*(X,Z) \ar[r]^\partial \ar[d]_{\secInd} & K_{*-1}^{\Gamma}(Z) \ar[d]_{\Ind} \\ 
K_*(C^*(Z)^\Gamma) \ar[r] \ar@{=}[d] & K_*(C^*(X)^\Gamma)\oplus K_*(D^*(Z)^\Gamma) \ar@{.>}[d]_{\id\oplus 0} \ar[r] & K_*(D^*(X,Z)^\Gamma) \ar[d]_{\sF_*} \ar[r]^\delta & K_{*-1}(C^*(Z)^\Gamma)\ar@{=}[d] \\
K_*(C^*(Z)^\Gamma) \ar[r] & K_*(C^*(X)^\Gamma) \ar[r] & K_*(C^*(X,Z)^\Gamma) \ar[r]^\partial & K_{*-1}(C^*(Z)^\Gamma)
}}
where the dotted vertical arrows indicate that the small squares containing them do \emph{not} commute due to the presence of the summand $K_*(D^*(Z)^\Gamma)$, and the third vertical composition is the relative index map $\Ind_\mathrm{rel}$. 

\section{Relative cycles from the Dirac operator}\label{SectionBVPs}

This section contains a minimalist review of how to obtain relative cycles from Dirac operators on manifolds with boundary by choosing a boundary condition. While the use of boundary conditions is not necessary to define the relative class of the Dirac operator on a manifold with boundary, we will find them to be a convenient computational tool. Relative cycles have been constructed from boundary conditions for first-order operators on compact manifolds by Baum, Douglas and Taylor \cite{BDT1989}. Extensions to first-order operators on non-compact manifolds, and to the minimal boundary condition for higher-order operators on (possibly non-compact) manifolds have been given by Forsyth--Goffeng--Messland--Rennie and Fries respectively \cite{Forsyth2019,Fries2025}. Recently, Engel and the author discussed general regular boundary conditions for operators of arbitrary order on possibly non-compact manifolds \cite{Engel2026}. We will recall only what is necessary for our application in \cref{SecAPS}.

Let $\olO$ be a complete spin manifold with boundary with interior $\Omega$. We will denote by $L^2(\olO)$ the space of $L^2$-sections of the spinor bundle, keeping the bundle implicit in our notation. We denote by $\rho$ the representation of $C_0(\olO)$ by multiplication operators. Suppose that $\olO$ carries a free, proper $\Gamma$-action by spin isometries. Consider the Dirac operator $\Dirac$. We will focus on the even-dimensional case, so that $\Dirac$ is a graded operator 
$$
\Dirac = \begin{pmatrix}
0 & \Dirac^- \\ \Dirac^+ & 0
\end{pmatrix} \; .
$$
The Dirac operator defines an unbounded operator on $L^2(\olO)$ with domain consisting of smooth sections with compact support in the interior. We consider the self-adjoint extension
\begin{equation}\label{bc}
  \Dirac_e = \begin{pmatrix}
0 & \Dirac^-_{\min} \\ \Dirac^+_{\max} & 0
\end{pmatrix} \; .
\end{equation}
Here, $\Dirac^-_{\min}$ denotes the closure of $\Dirac^-$ as an unbounded operator, and $\Dirac^+_{\max}$ the Hilbert space adjoint of $\Dirac_{\min}^-$. The domain of $\Dirac_{\min}^-$ is continuously included in the space $H^1_{0,\loc}(\olO)$ of local Sobolev sections vanishing on the boundary, making it a semi-regular boundary condition in the sense of \cite{Engel2026}. The Dirac operator is $\Gamma$-equivariant, and the same holds for $\Dirac_e$. Hence, the results of \cite{Engel2026} provide the following\footnote{This result can also be obtained from \cite{Forsyth2019}, specifically Section 2.2.2 therein, though a little analysis has to be done to verify that the extension $\Dirac_e$ falls in the setting of that work.}.

\begin{prop}[{\cite[Prop.~6.8]{Engel2026}}]
  Let $\olO$ be an even-dimensional, complete spin manifold with boundary. Suppose that $\olO$ carries a free, proper action of $\Gamma$ by spin isometries. Let $\chi$ be any normalizing function\footnote{A normalizing function is an odd, continuous function $\chi: \, \RR \to \RR$ with $\lim_{\lambda\to\infty} \chi(\lambda)=1$.}. Then, the triple
  $$
    \left( L^2(\olO), \rho, \chi(\Dirac_e) \right)
  $$
  is a relative equivariant Fredholm module over $(\olO,\dO)$, whose K-homology class is independent of $\chi$, and coincides with $[\Dirac] \in K^\Gamma_0(\Omega)$ under the excision isomorphism.
\end{prop}

In our geometric application in \cref{SecAPS} below, we will investigate the image of this relative class under the secondary image map in the presence of uniformly positive scalar curvature. Before that, recall that in \cref{CoarseEqMV} above, we noted that the algebraic Mayer-Vietoris boundary map $K_*(D^*(X,Z)^\Gamma) \to K_{*-1}(D^*(Z)^\Gamma)$ is zero if $Z\hookrightarrow X$ is a coarse equivalence. As a small demonstration of the machinery of \cref{SecSecInd}, we combine this fact with \cref{IndexCommAlgMV} to give an alternative proof of bordism invariance of the coarse index.

\begin{cor}\label{BordInvIndex}
  Let $\olO$ be a complete spin $\Gamma$-manifold with boundary, and suppose that $\dO \hookrightarrow \olO$ is a coarse equivalence of $\Gamma$-spaces. Then,
  $$
    \mathrm{Ind}^\Gamma(\Dirac_{\dO}) = 0 \; .
  $$
\end{cor}

\begin{proof}
  Consider the image of $[\Dirac_{\Omega}] \in K_*(\olO,\dO)$ under the maps of the diagram \\
  \centerline{\xymatrix{
    K_{*}^\Gamma(\olO,\dO) \ar[r]^\partial \ar[d]_{\secInd} & K_{*-1}^\Gamma(\dO) \ar[d]^-{\Ind} \\
    K_{*}\left( D^*(\olO,\dO)^\Gamma\right) \ar[r]_-\delta & K_{*-1}\left( C^*(\dO)^\Gamma \right)
  }} 
  Its image under the upper-horizontal map is $[\Dirac_{\dO}] \in K_{*-1}^\Gamma(\dO)$ by boundary-of-Dirac-is-Dirac\footnote{Using \cref{EquivKhom}, the equivariant case of boundary-of-Dirac-is-Dirac can be reduced to the non-equivariant case, which is discussed in e.g. \cite[~Prop. 11.2.5]{HigsonRoe2000}.}. Hence, going right then down computes $\mathrm{Ind}^\Gamma(\Dirac_{\dO})$. On the other hand, since we assumed that $\dO \hookrightarrow \olO$ is a coarse equivalence, the Mayer-Vietoris boundary map $\delta$ is zero, as discussed in \cref{CoarseEqMV}. Thus, going down then right gives zero.
\end{proof}

Now, we lay some groundwork for our application to scalar curvature in \cref{SecAPS}. Assume that $\olO$ has uniformly positive scalar curvature (upsc). By the Lichnerowicz formula, it then holds that
$$
  \|\Dirac^-u\| \geq c \|u\|
$$
for some $c>0$ and every smooth section $u$ with compact support in the interior. It follows that $\Dirac^+_{\max}\Dirac_{\min}^- \geq c^2$, and $\Dirac_{\min}^-\Dirac^+_{\max}$ can only have $0$ as spectral value with absolute value $<c^2$. We pick a normalizing function $\chi$ so that $\chi^2(\lambda)=1$ for all $|\lambda|\geq c$.

\begin{prop}\label{InvolutionModBdry}
  Suppose that $\olO$ has uniformly positive scalar curvature. Then, $\chi(\Dirac_e)$ is in $D^*(\olO)^\Gamma$, and an involution modulo $D^*(\dO\subseteq \olO,\dO)^\Gamma$.
\end{prop}

\begin{proof}
  The proof is analogous to the boundaryless case (see e.g. \cite[Section~2]{Roe2016}). This is sketched in \cite[Remark~6.16]{Engel2026} (see also preceeding results therein), and details are given in the uniform setting in \cite[Section~7.1]{LykoThesis}.
\end{proof}

\begin{rem}\label{rem_psc_subset}
  \cref{InvolutionModBdry} is analogous to the fact that if $M$ is a complete spin manifold without boundary that has upsc outside a subset $Z$, then $\chi(\Dirac_M)$ is an involution modulo $C^*(Z\subseteq M)$, see e.g. \cite[Section~2]{Roe2016}. Since $\Dirac^+_{\max}$ may have non-trivial kernel even in the presence of upsc, the boundary has to be treated as if it did not have positive scalar curvature. Indeed, \cref{InvolutionModBdry} can be straight-forwardly generalized to the situation where $\olO$ only has upsc outside some closed $Z\subseteq\olO$, in which case $\chi(\Dirac_e)$ is an involution modulo $D^*(\dO\cup Z \subseteq \olO,\dO)$. We will not need this, however.
\end{rem}

Thanks to \cref{InvolutionModBdry}, if $\olO$ has uniformly positive scalar curvature, $\chi(\Dirac_e)$ defines an element in the K-theory of $D^*(\olO)^\Gamma/D^*(\dO\subseteq\olO,\dO)$. We denote its image under the boundary map
\begin{equation}\label{refined_index}
  K_1\left( \frac{D^*(\olO)^\Gamma}{D^*(\dO\subseteq\olO,\dO)^\Gamma} \right)  \xlongrightarrow{\partial} K_0\left( D^*(\dO\subseteq\olO,\dO)^\Gamma \right) \cong K_0\left( D^*(\dO)^\Gamma \right)
\end{equation}
by $\Ind_{\dO}(\Dirac)\in K_0( D^*(\dO)^\Gamma )$. By naturality of the boundary map, there is a commutative square \\
\centerline{\xymatrix{
  K_1\left( \frac{D^*(\olO)^\Gamma}{D^*(\dO\subseteq\olO,\dO)^\Gamma} \right) \ar[r] \ar[d]_{\Ind_{\dO}} & K^\Gamma_0(\olO,\dO) \ar[d]^{\secInd} \\
  K_0\left( D^*(\dO)^\Gamma \right) \ar[r] & K_0\left( D^*(\olO,\dO)^\Gamma \right)
}}
so that $\Ind_{\dO}(\Dirac)$ refines $\secInd(\Dirac)$.

\section{Application to positive scalar curvature}\label{SecAPS}

Let $\Sigma$ be an odd-dimensional complete spin manifold with a free and proper action of $\Gamma$ by isometries. If $\Sigma$ carries a upsc-metric, then $\Dirac_\Sigma$ is invertible. In this case, the spectral projection $P_>:=\mathbf{1}_{>0}(\Dirac_\Sigma)$ is a projection in $D^*(\Sigma)^\Gamma$ that provides a lift for $[\Dirac_\Sigma] \in K_1^\Gamma(\Sigma)$ along
$$
  K_0\left(D^*(\Sigma)^\Gamma\right) \longrightarrow K_0\left( \frac{D^*(\Sigma)^\Gamma}{C^*(\Sigma)^\Gamma} \right) \cong K_1^\Gamma(\Sigma) \; .
$$

\begin{defn}
  The \emph{higher $\rho$-invariant} $\rho^\Gamma(\Dirac_\Sigma)$ is defined as the class of $P_>$ in $K_0(D^*(\Sigma)^\Gamma)$.
\end{defn}

Now, suppose that $\Sigma$ is the boundary of an even-dimensional spin manifold carrying a free and proper action of $\Gamma$ by isometries that extends the $\Gamma$-action on $\Sigma$. Suppose that $\olO$ has product form near the boundary. Then, we may attach the semi-infinite cylinder $\Sigma_\infty := \RR_{\leq 0}\times\Sigma$ (endowed with the product metric) to form the complete, boundaryless manifold $M:=\Sigma_\infty \cup_\Sigma \olO$. Since $\Sigma$ has upsc, $M$ has upsc away from $\olO$. Specifically, suppose the scalar curvature on the complement of $\olO$ is $\geq\frac{c^2}{4}$ for some $c>0$. Let $\chi$ be a normalizing function with $\chi(\lambda)^2=1$ for all $|\lambda|\geq c$. Then, $\chi(\Dirac_M)\in D^*(M)^\Gamma$ is an involution modulo $C^*(\olO \subseteq M)$ (see e.g. \cite[Section~2]{Roe2016}, also compare \cref{rem_psc_subset}). Hence, it provides a lift of the class $[\Dirac_M] \in K^\Gamma_0(M)$ to $K_1(D^*(M)^\Gamma/C^*(\olO\subseteq M)^\Gamma)$. 

\begin{defn}
  The \emph{delocalized APS-index} $\Ind_\mathrm{APS}(\Dirac_\Omega)$ is defined as the image of $[\chi(\Dirac_M)]$ under the boundary map
    $$
      K_1\left(\frac{D^*(M)^\Gamma}{C^*(\olO\subseteq M)^\Gamma}\right) \xlongrightarrow{\partial} K_0\left( C^*(\olO\subseteq M)^\Gamma \right) \cong K_0\left( C^*(\olO)^\Gamma \right) \; .
    $$
\end{defn}

Both $K_0(D^*(\Sigma)^\Gamma)$ and $K_0(C^*(\olO)^\Gamma)$ admit maps into $K_0(D^*(\olO)^\Gamma)$ induced by the respective inclusion. We will keep these maps implicit in our notation. The delocalized APS-index theorem states that
$$
  \Ind_\mathrm{APS}(\Dirac_\Omega) = - \rho^\Gamma(\Dirac_\Sigma) \; \in \; K_0\left(D^*(\olO)^\Gamma\right) \; .
$$
In the even-dimensional case (which we work in), it was proved by Piazza and Schick \cite{Piazza2014}. The general case was established by Xie and Yu using localization algebra techniques \cite{Xie2014}. Both Piazza-Schick and Xie-Yu assume the action of $\Gamma$ on $\olO$ to be cocompact. We note that the negative sign in the index formula is not present in those works, see \cref{SignAPS} below.

\begin{thm}\label{RelIndRho}
Let $\olO$ be an even-dimensional, collared spin manifold carrying a free and proper action of $\Gamma$ by spin isometries.
\begin{itemize}
  \item[(i)] If $\olO$ has uniformly positive scalar curvature, then\footnote{Recall that $\Ind_{\dO}$ denotes the refinement of the secondary index map defined in \eqref{refined_index}, and that $\Dirac_\Omega$ is in its domain, since the upsc-assumption implies $\chi(\Dirac_\Omega)$ to be an involution modulo $D^*(\dO\subseteq\olO,\dO)$ by \cref{InvolutionModBdry}.}
  $$
    \Ind_{\dO}(\Dirac_\Omega) = \rho^\Gamma(\Dirac_{\dO}) \; \in \; K_{0}\left( D^*(\dO)^\Gamma\right) \; .
  $$
  \item[(ii)] If $\olO$ only has uniformly positive scalar curvature near the boundary, then
  $$
    \secInd(\Dirac_\Omega) = \Ind_{APS}(\Dirac_\Omega) + \rho(\Dirac_{\dO}) \; \in \; K_{0}\left( D^*(\olO,\dO)^\Gamma\right) \; .
  $$
\end{itemize}
\end{thm}

Before we come to the proof of \cref{RelIndRho}, we deduce the delocalized APS-index theorem as a corollary. Note we manage without the assumption of a cocompact $\Gamma$-action.

\begin{cor}
  Let $\olO$ be as in \cref{RelIndRho}.
  \begin{itemize}
    \item[(i)] If $\olO$ has uniformly positive scalar curvature, then
    $$
      \rho^\Gamma(\Dirac_{\dO}) = 0 \; \in \; K_{0}\left( D^*(\olO)^\Gamma\right) \; .
    $$
    \item[(ii)] If $\olO$ only has uniformly positive scalar curvature near the boundary, then
    $$
      \Ind_{APS}(\Dirac_\Omega) = - \rho(\Dirac_{\dO}) \; \in \; K_{0}\left( D^*(\olO)^\Gamma\right) \; .
    $$
\end{itemize}
\end{cor}

\begin{proof}
  The secondary index maps $\Ind_{\dO}$ and $\secInd$ are defined as K-theoretic boundary maps on the quotient of $D^*(\olO)^\Gamma$ by the respective ideal. By exactness, their images lie in the kernel of the respective inclusion maps into $K_0(D^*(\olO)^\Gamma)$. Hence, $\Ind_{\dO}(\Dirac_\Omega)$ and $\secInd(\Dirac_\Omega)$ vanish upon being pushed into $K_0(D^*(\olO)^\Gamma)$, so that the claim follows from the respective formulas in \cref{RelIndRho}.
\end{proof}

\begin{rem}\label{SignAPS}
  In other treatments (\cite{Piazza2014,Xie2014,Zeidler2016,Zenobi2017,Bunke2018}), the delocalized APS-index theorem does not have the sign that ours does. This is a consequence of a difference in choice of orientation. Indeed, in the other treatments of this theorem the orientation on $M$ is chosen such that the attached cylinder is positively oriented, meaning that the orientation on $\Omega$ is reversed. While this results in a sign-free delocalized APS-index theorem, it is the less natural choice from our point of view. Since our primary interest was the relative index $\secInd(\Dirac_\Omega)$, it made sense for us to choose that orientation on $M$ which coincides with the given one on $\Omega$. This resulted in the sign.
\end{rem}

The remainder of this section is dedicated to the proof of \cref{RelIndRho}. We approach it via reduction to an explicit computation on a cylinder. Hence, we will begin with an analysis of that case. We denote the boundary $\dO$ by $\Sigma$. Recall that we set $P_>:=\mathbf{1}_{>0}(\Dirac_\Sigma)$, and set $P_<:=\mathbf{1}_{<0}(\Dirac_\Sigma)$.

\begin{lem}\label{KernelCylinder}
  Let $\Sigma$ be an odd-dimensional complete spin manifold with uniformly positive scalar curvature. Consider the operator $D = \partial_t + \Dirac_\Sigma$ on the cylinder $\RR_{\geq 0}\times \Sigma$. It holds that
  $$
    \ker(D_{\max}) = \left\{ u(t,x) = e^{-t\Dirac_\Sigma}u_0 \, | \, u_0\in P_>L^2(\Sigma) \right\} \; .
  $$
\end{lem}

\begin{proof}
  We have to find all $u\in L^2(\RR_{\geq 0}\times\Sigma)$ such that $(\partial_t + \Dirac_\Sigma)u=0$ in the sense of distributions. Due to interior regularity for the elliptic operator $D=\partial_t + \Dirac_\Sigma$, any such $u$ must be smooth in the interior. From this, one finds $u_t = u(t,-)$ to be a well-defined element of $\dom(\Dirac_\Sigma)$ for all $t>0$. We wish to check that the limit $u_0:=\lim_{t\to 0}u_t$ exists. Once this is done, standard results on Cauchy problems in Hilbert spaces (see e.g. \cite[~Prop. 6.6]{Schmudgen2012}) yield that $u_t = e^{-t\Dirac_\Sigma}u_0$, which can only be square-integrable if $P_<u_0=0$.

  Existence of the $L^2$-limit can likely be derived from standard trace theorems (e.g. \cite[~Chapter 1]{Lions1972}). However, we will take a more pedestrian approach making use of the spectral gap for $\Dirac_\Sigma$\footnote{The outline of the following arguments was developed over the course of a conversation with Google Gemini 3.5 Flash.}. To that end, observe that 
  $$
    \frac{d}{dt} \| P_< u_t \|^2 = - 2 \langle P_< u_t, \Dirac_\Sigma P_< u_t \rangle \geq 2c \|P_< u_t\|^2 \; ,
  $$
  where we used $|\Dirac_\Sigma|\geq c>0$. Hence, $\| P_< u_t \|^2$ is a non-decreasing function for $t>0$. Since $u$ is square-integrable, this can only be if $P_<u_t = 0$ for all $t>0$. Hence, 
  $$
    \frac{d}{dt} \| u_t \|^2 = - 2 \langle u_t, \Dirac_\Sigma u_t \rangle \leq -2c \|u_t\|^2 \; .
  $$
  It follows that $\|u_t\|^2$ has a finite limit as $t\to 0$: Since $\|u_t\|^2$ is non-increasing, it suffices to show that $\|u_t\|^2$ is bounded for $t\to 0$. Fix some $t_0 > 0$. Then,
  $$
    \|u_t\|^2- \|u_{t_0}\|^2  =  -\int_{t}^{t_0} \frac{d\|u_s\|^2}{ds} \, ds  \leq 2c \int_t^{t_0} \|u_s\|^2 \, ds  \leq 2c\|u\|^2 < \infty \; ,
  $$
  where Tonelli's theorem justifies the final estimate. This proves boundedness.
  
  We use this to show that $u_t$, $t>0$, is a Cauchy net. Let $t>s>0$. It holds that $u_t = e^{-(t-s)\Dirac_\Sigma}u_s$, from which we may conclude that $\langle u_t, u_s \rangle = \|u_{\frac{t-s}{2}}\|^2$. Thus
  $$
    \|u_t - u_s\|^2 = \|u_t\|^2 - 2\|u_{\frac{t-s}{2}}\|^2 + \|u_s\|^2 \xrightarrow{t,s\to 0} 0\; .
  $$
  Hence, the limit $u_0 = \lim_{t\to 0}u_t$ exists, which, as discussed above, completes the proof.
\end{proof}

Having identified this kernel, we can prove \cref{RelIndRho} (i) in the case of a cylinder by direct computation.

\begin{lem}\label{FormulaCylinder}
  Let $\Sigma$ be an odd-dimensional complete spin manifold with a free, proper, isometric $\Gamma$-action and uniformly positive scalar curvature. Consider $\olC:= \RR_{\geq 0}\times \Sigma$ with the product metric, and slice-wise $\Gamma$-action. Then,
  $$
    \Ind_{\Sigma}(\Dirac_C) = \rho^\Gamma(\Dirac_{\Sigma}) \; \in \; K_{0}\left( D^*(\Sigma)^\Gamma\right) \; .
  $$
\end{lem}

\begin{proof}
  As described in \cref{SectionBVPs}, the relative K-homology class of 
  $$
    \Dirac_\Omega = \Dirac = \begin{pmatrix}
    0 & \Dirac^- \\ \Dirac^+ & 0
    \end{pmatrix} = \begin{pmatrix}
    0 & -\partial_t + \Dirac_\Sigma \\ \partial_t + \Dirac_\Sigma & 0
    \end{pmatrix}
  $$ 
  can be obtained from the boundary condition
  $$
    \Dirac_e = \begin{pmatrix}
    0 & \Dirac^-_{\min} \\ \Dirac_{\max}^+ & 0
    \end{pmatrix} \; .
  $$
  Indeed, by assumption of upsc for $\Sigma$ (and hence also for $\olC$), there is $c>0$ such that $\Dirac^+_{\max}\Dirac_{\min}^-\geq c^2$ and $\Dirac_\Sigma^2\geq c^2$. Pick a normalizing function $\chi$ such that $1-\chi^2$ is supported inside $(-c,c)$. Write $\chi(\Dirac_e)=\begin{pmatrix} 0 & U^* \\ U & 0 \end{pmatrix}$. Then, 
  $$
    \begin{pmatrix}
    U^*U & 0 \\ 0 & UU^*
    \end{pmatrix} = \chi^2(\Dirac_e) = I - P_{\ker(\Dirac_e)} = \begin{pmatrix}
    I-P_{\ker(\Dirac^+_{\max})} & 0 \\ 0 & I
    \end{pmatrix}  \; .
  $$
  We conclude that $U$ is a partial isometry with kernel $\ker(\Dirac_{\max}^+)^\perp$ and trivial cokernel. Moreover, $I-\chi(\Dirac_e)^2 \in D^*(\Sigma\subseteq\olC,\Sigma)^\Gamma$, meaning that the projection $P_{\max}$ onto $\ker(\Dirac^+_{\max})$ is in $D^*(\Sigma\subseteq\olC,\Sigma)^\Gamma$. Indeed, the K-theory class $[P_{\max}] \in K_0(D^*(\Sigma\subseteq\olC,\Sigma)^\Gamma)$ represents the image of $[U] \in K_1(D^*(\olC)^\Gamma/D^*(\Sigma\subseteq\olC,\Sigma)^\Gamma)$ under the boundary map
  $$
    K_1\left(\frac{D^*(\olC)^\Gamma}{D^*(\Sigma\subseteq\olC,\Sigma)^\Gamma}\right) \xlongrightarrow{\partial} K_0\left(D^*(\Sigma\subseteq\olC,\Sigma)^\Gamma\right) \; .
  $$
  Up to the identifiction $K_0(D^*(\Sigma\subseteq\olC,\Sigma)^\Gamma)\cong K_0(D^*(\Sigma)^\Gamma)$ via \cref{RelLocalizedAisZ}, this boundary map is the secondary index map $\Ind_\Sigma$. Since $P_>$ represents $\rho^\Gamma(\Dirac_\Sigma)$, we thus have to show that $[P_{\max}] = [P_{>}]$ under this identifiction. 
  
  In \cref{KernelCylinder} it was shown that $\ker(\Dirac_{\max}^+)=\ker((\partial_t + \Dirac_\Sigma)_{\max})$ consists precisely of $u\in L^2(\olC)$ with $u(t,-)=e^{-t\Dirac_\Sigma}u_0$ and $u_0 \in P_>L^2(\Sigma)$. The remaining proof consists of turning this connection into an equality of K-theory classes. To that end, we shall first show that, up to homotopy, it suffices to consider the image of $P_{\max}$ up to some finite time $t=R$. Once this is done, we can homotope $t$ to zero, and identify the resulting projection with $P_>$.
  
  Fix some small $\varepsilon>0$. Because $P_{\max} \in D^*(\Sigma\subseteq \olC,\Sigma)^\Gamma$ is the limit of operators supported near $\Sigma$, there exists some $R\geq 0$ such that $\|P_{\max}u\|<\varepsilon\|u\|$ whenever $d(\supp(u),\Sigma) \geq R$. Then, let $P_R$ be the projection onto the subspace containing all functions $u(t,\cdot)=\psi_R\cdot e^{-t\Dirac_\Sigma}u_0$ for $u_0\in P_>L^2(\Sigma)$, where $\psi_R=\mathbf{1}_{[0,R]}(t)$. Then 
  $$
    \|(P_{\max} - P^R)u\|=\|P_{\max}u\|<\varepsilon\|u\|
  $$
  if $d(\supp(u),\Sigma)\geq R$. Suppose on the other hand that $\supp(u)\subseteq [0,R]\times\Sigma$. Write $u=\psi_R e^{-t\Dirac_\Sigma}u_0 + u^\perp$, where $u^\perp$ is uniquely determined by the requirement that $\langle \psi_R e^{-t\Dirac_\Sigma} v_0,u^\perp\rangle=0$ for all $v_0\in P_>L^2(\Sigma)$. Note that in particular $u^\perp=\psi_R u^\perp$, so that it follows that $\langle e^{-t\Dirac_\Sigma} v_0,u^\perp\rangle=0$ for all $v_0\in P_>L^2(\Sigma)$ as well. Thus
  \begin{align*}
    \left\| (P_{\max} - P^R)u \right\| &= \left\| (P_{\max} - I)\psi_R e^{-t\Dirac_\Sigma}u_0 \right\| \\
    &= \left\| P_{\max} e^{-t\Dirac_\Sigma}u_0 - P_{\max}\left( (1-\psi_R)e^{-t\Dirac_\Sigma}u_0 \right) - \psi_R e^{-t\Dirac_\Sigma}u_0 \right\| \\
    &= \left\| e^{-t\Dirac_\Sigma}u_0 - P_{\max}\left( (1-\psi_R)e^{-t\Dirac_\Sigma}u_0 \right) - \psi_R e^{-t\Dirac_\Sigma}u_0 \right\| \\
    &\leq 2\cdot \left\| (1-\psi_R)e^{-t\Dirac_\Sigma}u_0 \right\| \; .
  \end{align*}
  It holds that 
  $$
  \frac{\varepsilon^2}{4} \int_{t=0}^R e^{-2t\lambda} \, dt \geq \int_{t=R}^\infty e^{-2t\lambda} \, dt
  $$
  for all $\lambda\geq c$ whenever $R$ is sufficiently large, say $R\gsim \log(\varepsilon^{-2}) c^{-1}$. Letting $E_\lambda$ denote the spectral measure of $\Dirac_\Sigma$, we go on to compute that
  \begin{align*}
  \left\| (1-\psi_R)e^{-t\Dirac_\Sigma}u_0 \right\|^2 &=  \int_{t=R}^\infty \left\| e^{-t\Dirac_\Sigma}u_0 \right\|^2 \, dt \\
  &= \int_{t=R}^\infty \left( \int_{\lambda=c}^\infty e^{-t\lambda} d\langle E_\lambda u_0,u_0\rangle \right) \, dt \\
  &= \int_{\lambda=c}^\infty \left( \int_{t=R}^\infty e^{-t\lambda}   \, dt \right) d\langle E_\lambda u_0,u_0\rangle \\
  &\leq \frac{\varepsilon^2}{4} \int_{\lambda=c}^\infty \left( \int_{t=0}^R e^{-t\lambda}   \, dt \right) d\langle E_\lambda u_0,u_0\rangle \\
  &= \frac{\varepsilon^2}{4} \left\| \psi_R e^{-t\Dirac_\Sigma}u_0 \right\|^2 \\
  &\leq \frac{\varepsilon^2}{4} \|u\|^2 \; .
  \end{align*}
  The interchange of order of integration is justified by Tonelli's theorem. Thus
  $$
  \left\| (P_{\max} - P^R)u \right\| \leq \varepsilon \|u\| \; .
  $$
  Overall, we conclude that $P_{\max}$ and $P_R$ are projection whose difference has norm strictly less than $1$, so that they are homotopic (see for example \cite[Proposition 4.1.7]{HigsonRoe2000}).

  Let $\olC_R=[0,R]\times \Sigma$, and denote by $V_R$ the inclusion map $L^2(\olC_R)\hookrightarrow L^2(\olC)$. Also consider the isometry $V_\Sigma: \, L^2(\Sigma) \to L^2(\olC_R)$, $(V_\Sigma v)(t,x)=R^{-1/2}v(x)$. The conjugation maps corresponding to these isometries induce the isomorphisms\footnote{More precisely, the isometry $V_\Sigma$ uniformly covers the inclusion $\Sigma \hookrightarrow [0,R]\times\Sigma$, and $V_R$ uniformly covers $[0,R]\times\Sigma \hookrightarrow \olC$. For the sake of brevity, we have eschewed uniform covering isometries in our discussion of functoriality of (relative) structure algebras, so the reader is directed to \cite[~Section 4.3]{Siegel2012}.}
  $$
  K_*\left(D^*(\Sigma)^\Gamma \right) \xlongrightarrow{\sim} K_*\left( D^*(\Sigma\subseteq\olC_R,\Sigma)^\Gamma \right) \xlongrightarrow{\sim} K_*\left( D^*(\Sigma \subseteq \olC,\Sigma)^\Gamma \right) \; .
  $$
  Observe that $P_R$ is in the image of $\Ad(V_R)$. Namely, it is the image of the projection onto the subspace of $L^2(\olC_R)$ containing all functions $u(t,\cdot)=e^{-t\Dirac_\Sigma}u_0$, $u_0 \in P_>H_\Sigma$, $t\in [0,R]$. Denote that projection by $P_R$ as well. Now, there is a homotopy of projections $P_R^s$, $s\in[0,1]$, where $P_R^s$ projects onto the subspace of $H_R$ containing all functions $u(t,\cdot)=e^{-(1-s)t\Dirac_\Sigma}u_0$, $u_0 \in\im(P_0)$. Then, $P_R^0=P_R$, while $P_R^1=V_\Sigma P_> V^*_\Sigma$. Because $V_\Sigma$ and $V_R$ induce isomorphisms, we conclude that
  $$
  [P_{\max}]=[P_R]= [P_>] \; \in \; K_*\left(D^*(\Sigma)^\Gamma \right) \cong K_*\left( D^*(\Sigma \subseteq \olC,\Sigma)^\Gamma \right) \; .
  $$
  Since we have already noted that the former projection represents $\Ind_\Sigma(\Dirac)$ and the latter $\rho^\Gamma(\Dirac_\Sigma)$, this concludes the proof.
\end{proof}

From here, we apply our understanding of the cylinder to the general case. Recall that we formed the manifold $M$ by attaching the cylinder $\Sigma_\infty = \RR_{\leq 0}\times \Sigma$ to the boundary of $\olO$. We let $\psi_{\olO}$ and $\psi_{\Sigma_\infty}$ denote the respective normalizing functions. Moreover, we fix a characteristic function $\chi$ with $\chi^2(\lambda) = 1$ for $|\lambda|\geq c$.

\begin{lem}\label{decomp}
  Assume that $\Sigma$ has uniformly positive scalar curvature.
  \begin{itemize}
      \item[(i)] $\psi_{\olO} \chi(\Dirac_M) \psi_{\Sigma_\infty}$ and $\psi_{\Sigma_\infty} \chi(\Dirac_M) \psi_{\olO}$ are in $D^*(\Sigma\subseteq M,\Sigma)^\Gamma$.
      \item[(ii)] $\psi_{\olO}\chi(\Dirac_M)\psi_{\olO}$ represents $[\Dirac_\Omega] \in K^\Gamma_0(\olO,\Sigma)$.
      \item[(iii)] $\psi_{\Sigma_\infty}\chi(\Dirac_M)\psi_{\Sigma_\infty}$ represents the class of $\Dirac_{\Sigma_\infty}$ in $K^\Gamma_0(D^*(\Sigma_\infty)^\Gamma/D^*(\Sigma\subseteq\Sigma_\infty,\Sigma)^\Gamma)$.
      \item[(iv)] If $\olO$ has uniformly positive scalar curvature, then $\psi_{\olO}\chi(\Dirac_M)\psi_{\olO}$ represents the class of $\Dirac_\Omega$ in $K^\Gamma_0(D^*(\olO)^\Gamma/D^*(\Sigma\subseteq\olO,\Sigma)^\Gamma)$.
  \end{itemize}
\end{lem}

Statement and proof of this lemma have significant overlap with \cite[Section~4.2]{Piazza2014}.

\begin{proof}
  Ad (i): We prove the claim for $\psi_{\olO} \chi(\Dirac_M) \psi_{\Sigma_\infty}$, that for $\psi_{\Sigma_\infty} \chi(\Dirac_M) \psi_{\olO}$ then follows by taking the adjoint. Approximating $\chi$ by normalizing functions with compactly supported Fourier transforms, one may realize $\psi_{\olO} \chi(\Dirac_M) \psi_{\Sigma_\infty}$ as the limit of finite-propagation operators supported near $\Sigma$. To see that $\psi_{\olO} \chi(\Dirac_M) \psi_{\Sigma_\infty}$ is locally compact over $M\setminus \Sigma$, let $f\in C_0(M\setminus\Sigma)$. We may assume that $\supp(f)$ has some positive distance $\delta>0$ from $\Sigma$. For any normalizing function $\chi'$ with $\chi'(\Dirac_M)$ having propagation $<\delta$, it holds that $\left(\psi_{\olO} \chi'(\Dirac_M) \psi_{\Sigma_\infty}\right) f=0$. On the other hand, $\chi(\Dirac_M)$ and $\chi'(\Dirac_M)$ differ by an operator in $C^*(M)^\Gamma$, see e.g. \cite[Lemma 10.6.3]{HigsonRoe2000}. Thus $\left(\psi_{\olO} \chi(\Dirac_M) \psi_{\Sigma_\infty}\right) f$ is compact. Analogously, $f\left(\psi_{\olO} \chi'(\Dirac_M) \psi_{\Sigma_\infty}\right)$ is also compact, whence we conclude $\psi_{\olO} \chi(\Dirac_M) \psi_{\Sigma_\infty}\in D^*(\Sigma\subseteq M,\Sigma)^\Gamma$.

  Ad (ii): The K-homology class of $\psi_{\olO}\chi(\Dirac_M)\psi_{\olO}$ represents the image of $[\Dirac_M]$ under the map restriction map $K_0^\Gamma(M) \to K_0^\Gamma(\Omega)\cong K^\Gamma_0(\olO,\Sigma)$. Hence, the claim follows from \cite[Prop.~10.8.8]{HigsonRoe2000}.

  Ad (iii): Analogously to (ii), $\psi_{\Sigma_\infty}\chi(\Dirac_M)\psi_{\Sigma_\infty}$ represents the class of $[\Dirac_{\Sigma_\infty}]$ in $K_0^\Gamma(\Sigma_\infty,\Sigma)$. Represent the latter by $\chi(\Dirac_{\Sigma_\infty,e})$ with $\Dirac_{\Sigma_\infty,e}$ the boundary condition \eqref{bc} for $\Dirac_{\Sigma_\infty}$. The operator $\chi(\Dirac_{\Sigma_\infty,e})$ is invertible modulo $D^*(\Sigma\subseteq\Sigma_\infty,\Sigma)^\Gamma$ by \cref{InvolutionModBdry}, hence it suffices to prove its difference from $\psi_{\Sigma_\infty}\chi(\Dirac_M)\psi_{\Sigma_\infty}$ to lie in $D^*(\Sigma\subseteq\Sigma_\infty,\Sigma)^\Gamma$. This can be done analogously to (i) by noting that modulo $D^*(\Sigma\subseteq\Sigma_\infty,\Sigma)^\Gamma$, we may replace $\chi$ by a normalizing function with compactly supported Fourier transform, which clearly leads to operators supported near $\Sigma$.

  Claim (iv) is proved entirely analogously to (iii).
\end{proof}

Now, all that remains to be done is assemble the pieces.

\begin{proof}[Proof of \cref{RelIndRho}]
  We focus on the proof of (ii). By \cref{decomp}(i) we may write
  $$
  \chi(\Dirac_M) = \begin{pmatrix}
  \psi_{\olO} \chi(\Dirac_M) \psi_{\olO} & 0 \\ 0 & \psi_{\Sigma_\infty} \chi(\Dirac_M) \psi_{\Sigma_\infty}
  \end{pmatrix} \; \mathrm{mod} \; D^*(\Sigma\subseteq M , \Sigma)^\Gamma \; .
  $$
  Here, $\psi_{\olO} \chi(\Dirac_M) \psi_{\olO} \in D^*(\olO)^\Gamma$ is an involution modulo $D^*(\olO,\Sigma)$, and $\psi_{\Sigma_\infty} \chi(\Dirac_M) \psi_{\Sigma_\infty} \in D^*(\Sigma_\infty)^\Gamma$ an involution modulo $D^*(\Sigma \subseteq \Sigma_\infty,\Sigma)^\Gamma$. Consider the commutative diagram\\
  \centerline{\xymatrix{
  K_{1}\left(\frac{D^*(\olO)^\Gamma}{D^*(\olO,\Sigma)^\Gamma}\right)\ar[d]_{\partial}\ar@{}[r]|{\oplus} & K_{1}\left(\frac{D^*(\Sigma_\infty)^\Gamma}{D^*(\Sigma\subseteq\Sigma_\infty,\Sigma)^\Gamma}\right)\ar[d]^{\partial} \ar[r] & K_{1}\left(\frac{D^*(M)^\Gamma}{D^*(\olO\subseteq M,\Sigma)^\Gamma}\right)\ar[d]^{\partial} \\
  K_{0}\left(D^*(\olO,\Sigma)^\Gamma\right)\ar@{}[r]|{\oplus} & K_{0}\left(D^*(\Sigma\subseteq \Sigma_\infty,\Sigma)^\Gamma\right) \ar[r] & K_{0}\left(D^*(\olO\subseteq M,\Sigma)^\Gamma\right)
  }} 
  where the horizontal arrows are the sums of the respective inclusion maps. Note that the first vertical arrow is precisely the secondary uniform index map. Up to identifying $K_{0}(D^*(\Sigma\subseteq \Sigma_\infty,\Sigma)^\Gamma) \cong K_0(D^*(\Sigma)^\Gamma)$, the second is the index map $\Ind_\Sigma$. The third vertical arrow maps the class of $\chi(\Dirac_M) \in D^*(M)^\Gamma$ to the image of $\Ind_\mathrm{APS}(\Dirac_M)$ under the inclusion map $C^*(\olO\subseteq M) \hookrightarrow D^*(\olO\subseteq M,\Sigma)$. Moreover, we may identify $K_{0}\left(D^*(\olO\subseteq M,\Sigma)^\Gamma\right)$ with $K_{0}\left(D^*(\olO,\Sigma)^\Gamma\right)$ by \cref{RelLocalized}. Appealing to \cref{decomp} (ii) and (iii), we thus find that
  $$
  \Ind_\mathrm{APS}(\Dirac_M) = \secInd \left( \Dirac_\Omega \right) + \Ind_\Sigma \left( \Dirac_{\Sigma_\infty} \right) \; ,
  $$
  To conclude the proof, we merely note that $\Ind_\Sigma( \Dirac_{\Sigma_\infty}) = - \rho^\Gamma(\Dirac_\Sigma)$ by \cref{FormulaCylinder}, where the sign comes from the fact that $\Sigma_\infty = (-\infty,0]\times\Sigma$ carries the opposite orientation from the cylinder $\olC=[0,\infty) \times \Sigma$ considered in that lemma.

  By using that if $\olO$ has uniformly positive scalar curvature, $\psi_{\olO} \chi(\Dirac_M) \psi_{\olO} \in D^*(\olO)^\Gamma$ is actually an involution modulo $D^*(\Sigma\subseteq\olO,\Sigma)$, and appealing to \cref{decomp} (iv), claim (i) is proved entirely analogously.
\end{proof}

\printbibliography[heading=bibintoc]

\end{document}